\documentclass[11pt]{amsart}
\usepackage{amssymb,amsthm,amsmath,amstext}
\usepackage{mathrsfs}
\usepackage{bm}
\usepackage{multirow}
\usepackage[all]{xy}

\usepackage[left=1.25in,top=1.35in,right=1.25in,bottom=1.35in]{geometry}

\usepackage{color}
\definecolor{orange}{RGB}{255,127,0}

\theoremstyle{plain}
\newtheorem{theorem}{Theorem}\newtheorem*{theorem*}{Theorem}
\newtheorem*{conjecture*}{Conjecture}
\newtheorem{conjecture}[theorem]{Conjecture}
\newtheorem{prop}[theorem]{Proposition}
\newtheorem{lemma}[theorem]{Lemma}
\newtheorem{cor}[theorem]{Corollary}

\theoremstyle{definition}
\newtheorem{defin}[theorem]{Definition}

\theoremstyle{remark}
\newtheorem{remark}[theorem]{Remark}

\newcommand{\linedef}[1]{\emph{#1}}
\newcommand{\cat}[1]{\mathsf{#1}}

\DeclareMathOperator{\lra}{\longrightarrow}
\DeclareMathOperator{\ra}{\rightarrow}
\DeclareMathOperator{\Db}{{\cat{D}^{\mathrm{b}}}}

\DeclareMathOperator{\Fl}{{\rm Fl}}
\DeclareMathOperator{\Gr}{{\rm Gr}}

\DeclareMathOperator{\Hom}{{\rm Hom}}
\DeclareMathOperator{\Ext}{{\rm Ext}}

\DeclareMathOperator{\Pic}{{\rm Pic}}

\newcommand{\ZZ}{\mathbb{Z}}
\newcommand{\QQ}{\mathbb{Q}}

\newcommand{\CC}{\mathbb{C}}

\newcommand{\LL}{\mathbb{L}}
\newcommand{\PP}{\mathbb{P}}

\newcommand{\C}{\CC}

\newcommand{\sod}[1]{\langle #1 \rangle}

\newcommand{\sheaf}[1]{\mathcal{#1}}

\newcommand{\kd}{\sheaf{D}}
\newcommand{\ke}{\sheaf{E}}
\newcommand{\kf}{\sheaf{F}}
\newcommand{\kh}{\sheaf{H}}

\newcommand{\kl}{\sheaf{L}}
\newcommand{\km}{\sheaf{M}}
\newcommand{\kn}{\sheaf{N}}
\newcommand{\ko}{\sheaf{O}}
\newcommand{\kq}{\sheaf{Q}}
\newcommand{\kr}{\sheaf{R}}
\newcommand{\ku}{\sheaf{U}}

\newcommand{\OO}{\ko}
\newcommand{\W}{\bigwedge}

\newcommand{\cO}{\mathcal{O}}

\usepackage[backref=page]{hyperref}
\hypersetup{pdftitle={}}%
\hypersetup{pdfauthor={}}%
\hypersetup{colorlinks=true,linkcolor=blue,anchorcolor=blue,citecolor=blue}

\begin{document}

\title{Nested varieties of K3 type}

\author{Marcello Bernardara}
\address{Institut de Math\'ematiques de Toulouse ; UMR 5219 \\ %
UPS, F-31062 Toulouse Cedex 9, France}
\email{marcello.bernardara@math.univ-toulouse.fr}

\author{Enrico Fatighenti }
\address{Department of Mathematical Sciences\\
Loughborough University\\
  LE113TU, UK}
\email[E.~Fatighenti]{e.fatighenti@lboro.ac.uk}

\author{Laurent Manivel}
\address{Institut de Math\'ematiques de Toulouse ; UMR 5219 \\ %
UPS, F-31062 Toulouse Cedex 9, France}
\email{manivel@math.cnrs.fr}
\date{\today}

\maketitle

\begin{abstract}
Using geometrical correspondences induced by projections and two-steps flag varieties, and
a generalization of Orlov's projective bundle theorem, we relate the Hodge structures and
derived categories of subvarieties of different Grassmannians. We construct isomorphisms between Calabi-Yau subHodge structures of hyperplane sections of $\Gr(3,n)$ and those of other varieties arising from symplectic Grassmannian and/or congruences of lines or planes. Similar results hold conjecturally for Calabi-Yau subcategories: we describe in details the Hodge structures and give partial categorical results relating the K3 Fano hyperplane sections of $\Gr(3,10)$ to other Fano varieties such as the Peskine variety.
Moreover, we show how these correspondences allow to construct crepant categorical resolutions of the Coble cubics.
\end{abstract}

\section{Introduction}

Fano varieties of K3 type have recently been investigated because of their potential 
relations with hyperK\"ahler manifolds \cite{debarre-voisin,fm,imcy}. More generally, Fano varieties of Calabi-Yau type
are endowed with special Hodge structures which can sometimes be mapped, through adequate correspondences,
to auxiliary manifolds, or, more generally, used to obtain geometrical information on the variety,
either of cycle-theoretical nature (see \cite{hassett:special-cubics} for cubic fourfolds and \cite{fik-griffiths} for
Griffths groups) or on moduli spaces (see \cite{debarre-voisin}).
In some cases these manifolds are genuine K3 surfaces or Calabi-Yau
manifolds. However, in most cases there is no actual Calabi-Yau manifold, but rather a
noncommutative version,
and the Hodge structures and correspondences underlie special subcategories of derived categories. A typical example
is that of cubic fourfolds and their Kuznetsov categories \cite{kuz:4fold,at12}, which are subcategories of 
K3 type in their derived categories (conjectured to be of geometric origin only for 
rational cubics). In this case the special Hodge structure of the cubic fourfold can be 
transfered to its variety of lines on which it gives rise to a genuine symplectic structure \cite{beauville-donagi}.
Similar phenomena can be observed for the Debarre-Voisin fourfolds, whose symplectic 
structures are induced from special Hodge structures on certain hyperplane sections 
of Grassmannians \cite{debarre-voisin}. New examples include hyperplane sections
of symplectic Grassmannians \cite{fm}.

In this paper we explore such examples in a more general context, and
relate their Hodge structures to each other.
First of all, hyperplane sections of Grassmannians are known to provide examples of Fano varieties 
of Calabi-Yau type under
rather general hypotheses: this was observed by Kuznetsov \cite{kuz-fractional} at the categorical level, and we 
provide a Hodge-theoretic statement (Theorem \ref{thm:sect-of-G}) 
under slightly more general hypotheses. 
Then we transfer the resulting special Hodge structures to auxiliary varieties inside 
other Grassmannians, through two different types of basic operations: projections on the one 
hand, and jumps on the other hand, the latter being defined by the natural correspondences
afforded by two-steps flag manifolds. Our results are most precise for hyperplane sections 
of Grassmannians of three-planes, for which a projection induces an additional two-form, while 
a jump defines a congruence of lines. We obtain relations with natural auxiliary varieties
at several levels: for Hodge structures, sometimes for derived categories, and also in the
Grothendieck ring of varieties. One of the tools we use is an extension (Proposition \ref{prop:main-sod}) of the famous
structure theorem of Orlov for derived categories of smooth blow-ups, to maps whose
fibers can be projective spaces of two different dimensions. These kinds of results are of
independent interest and are probably known
to experts, but did not appear in the literature until \cite{Jiang-Leung}, where the case of
the projectivization of the cokernel of a map between two vector bundles is treated.

Congruences of lines defined by skew-symmetric three-forms were studied in \cite{dpfmr}, 
where the authors asked how to compute their Hodge numbers. These congruences are Fano 
varieties, which we prove to be prime of index three, and we explain how to deduce their 
Hodge numbers from those of hyperplane sections of Grassmannians, which are not difficult 
to compute. In the special case 
of forms in ten variables, the derived category of the Debarre-Voisin fourfold admits a 
K3 subcategory, which we call the Kuznetsov component. An additional player
is the Peskine variety in $\PP^9$, whose Hodge numbers we also determine: remarkably, its
Hodge structure exhibits not just one, but three Hodge substructures of K3 type. 
We prove (see Theorem \ref{K} for a more detailed statement):

\begin{theorem*}
For $Y\subset G(3,10)$ a very general hyperplane section, let $K$ denote the Hodge substructure of $H^{20}(Y,\CC)$ given by the vanishing cohomology. Then three copies of $K$ are contained 
in the cohomology of the associated congruence of lines $T\subset G(2,10)$ (resp. of the associated Peskine variety $P\subset\PP^9$). 
\end{theorem*}

Actually, these copies of $K$ constitute the essential part of the cohomology 
of both $T$ and $P$. 
Moreover, we conjecture that it should be possible to enhance these observations to the categorical
level: the derived category of the Peskine variety (resp. of the congruence of lines)
should be made of three copies of the Kuznetsov component plus $4$ (resp. $9$) exceptional
objects. We construct such exceptional objects explicitly (Propositions \ref{prop:exc-coll-on-T} and \ref{prop:exc-coll-on-P}).

Three-forms in nine variables are also remarkable because of their relations with 
Coble cubics of abelian surfaces. We conjecture that in this case, a crepant categorical
resolution of singularities of the Coble cubic defined by a congruence of lines could be deduced and admits a rectangular Lefschetz decomposition. Crepant categorical resolution of singularities 
have recently been investigated in several different contexts (see \cite{kuz-lefschetz, 
leuschke, lunts}). Here we construct geometric resolutions of 
singularities of the Coble cubics in terms of an extra skew-symmetric two form, 
and we finally deduce
(see Theorem \ref{prop:main-sod} for a more precise statement):

\begin{theorem*}
Coble cubics admit  weakly crepant categorical resolutions of singularities.
\end{theorem*}

\subsection*{Notations}

We use the following notations: for an integer $n$, $V_n$ is a complex vector space of dimension
$n$. The Grassmannian $\Gr(k,V_n)$ (or $\Gr(k,n)$ for short) 
parametrizes $k$-dimensional linear subspaces of $V_n$, and 
$\ku$ and $\kq$ are the tautological (rank $k$)
and quotient (rank $n-k$) bundles, respectively. Similar notations will be used for the 2-step flag
varieties $\mathrm{Fl}(k_1,k_2,V_n)$, where $\ku_{k_i}$ denotes the rank $k_i$ tautological
bundle. If the numerical values are unambiguous in the context, we will use shorthands
$\Gr$ and $\mathrm{Fl}$ to make the text more readable.

\smallskip

We will generally denote skew-symmetric $2$-forms by $\omega$ and  $3$-forms by
$\Omega$.

Given a set $\{\omega_1,\ldots,\omega_r\}$ of 
$r$ linearly independent skew-symmetric  $2$-forms on $V_n$, we will denote by $I_r\Gr(k,V_n)$, 
and call an $r$-th symplectic Grassmannian,
the subvariety of those $k$-spaces that are isotropic with respect to $\omega_1,\ldots,\omega_r$.

If these forms are general, since $I_r\Gr(k,V_n)$ can be seen as the zero locus of a general section of the globally generated vector bundle $(\bigwedge^2 \ku^*)^{\oplus r}$,  
it must be smooth of  dimension $k(n-k) - r \frac{k(k-1)}{2}$ (or empty). 

Notice that, if $k=2$, $\wedge^2 \ku^*$ is nothing but the Pl\"ucker line bundle, so that
$I_r G(2,V_n)$ is a $r$-iterated hyperplane section in the Pl\"ucker embedding.
For $r=2$ we get the bisymplectic Grassmannians that 
were considered in \cite{vlad-bi}.

\smallskip

Given a set $\{\Omega_1,\ldots,\Omega_r\}$ of 
$r$ linearly independent skew-symmetric  $3$-forms on $V_n$,  and $k \geq 3$, we will denote by $T_r\Gr(k,V_n)$, and call an \linedef{$r$-th 3-alternate congruence Grassmannian}, the subvariety of those $k$-spaces that are isotropic with respect to $\Omega_1,\ldots,\Omega_r$. Notice that, if $k=3$, 
$\wedge^3 \ku^*$ is nothing but the Pl\"ucker line bundle, so that
$T_r G(3,V_n)$ is a $r$-iterated hyperplane section in the Pl\"ucker embedding.
If $k < 3$, we will denote by $T_r \Gr(k,V_n)$ the set of those $k$-planes $U=\langle u_1,\ldots , u_k\rangle $ of $V_n$ such that the form $\Omega(u_1,\ldots , u_k,\bullet)$ is degenerate (where $\bullet$ stands for $3-k$ variables).

If $k\ge 3$ and the $\Omega_i$ are general, since $T_r\Gr(k,V_n)$ can be seen as the zero locus of a general section of $(\bigwedge^3 \ku^*)^{\oplus r}$, a globally generated vector bundle, it must be smooth of  dimension $k(n-k) - r {k \choose 3}$ (or empty). For $k=2$, $T_r\Gr(k,V_n)$ is the zero locus of a general section of $\kq^*(1)^{\oplus r}$. So it is $n-2$ dimensional for $r=1$ and $0$ dimensional for $r=2$.

We will be mostly interested in the case $k \leq 3$. We will use the following notation:

\begin{itemize}
\item $I_r(3,n):= I_r \Gr(3,V_n)$, which has expected dimension $3(n-r-3)$. 
\item[] We denote also $I(3,n):=I_1(3,n)$,
\item[]
\item $I_r(2,n):= I_r \Gr(2,V_n)$, the $r$-th iterated hyperplane section of $\Gr(2,V_n)$.
\item[] We denote also $I(2,n):=I_1(2,n)$,
\item[]
\item $T_r(3,n):= T_r \Gr(3,V_n)$, the $r$-th iterated hyperplane section of $\Gr(3,V_n)$.
\item[]We denote also $T(3,n):=T_1(3,n)$,
\item[]
\item $HI_r(3,n):= T_1 I_r \Gr(3,V_n)$, the hyperplane section of $I_r(3,n)=I_r\Gr(3,V_n)$.
\item[]We denote also $HI(3,n):=HI_1(3,n)$.
\item[]
\item $T(2,n):= T_1 \Gr(2,V_n)$. This is the scheme of planes $P=\langle p_1, p_2\rangle$ such that
the linear form $\Omega_1(p_1, p_2, \bullet)$ vanishes identically. It is the zero-locus of a section of 
$\kq^* (1)$, so the expected dimension is $n-2$.
\item[]
\item $P(1,n):= T_1 \Gr(1,V_n)$. This is the scheme of lines $L=\langle p\rangle$ such that the 
two form $\Omega_1(p, \bullet, \bullet)$ does not have maximal rank. If $\Omega_1$ is general, 
this is a codimension $3$ subvariety (smooth for $n\leq 10$) of $\PP^{n-1}$ if $n$ is even, or a 
hypersurface of degree $(n-3)/2$ (smooth for $n \leq 6$) in $\PP^{n-1}$ if $n$ is odd.
\end{itemize}

\section{Fano varieties of Calabi-Yau type and sections of Grassmannians}

\subsection{Definitions}
Fano varieties of Calabi-Yau type are the main subject of this paper. The definition
of such varieties (Definition \ref{def:CYtype}) is of Hodge-theoretical nature.
For a complete introduction to Hodge theory, the reader can refer to \cite{voisin-book}.

\begin{defin}\label{def:CYtype}
Let $X$ be a smooth, projective $n$-dimensional Fano variety and $j$ be a non-negative integer. The cohomology group $H^j(X, \C) \cong \bigoplus_{p+q=j} H^{p,q}(X)$ (with $j \geq k$) is said to be of $k$ Calabi-Yau type if
\begin{itemize}
\item $h^{\frac{j+k}{2},\frac{j-k}{2}}(X)=1$;
\item $h^{p,q}(X)=0$, for all $p+q=j, \ p <\frac{k+j}{2}$.
\end{itemize} 
Moreover, $X$ is said to be of $k$ (pure) Calabi-Yau type ($k$--FCY or Fano of $k$-CY type for short) if there exists at least a positive integer $j$ such that $H^j(X, \C)$ is of $k$ Calabi-Yau type. Similarly, $X$ is said to be of mixed $(k_1, \ldots, k_s)$ Calabi-Yau type if the cohomology of $X$ has different level CY structures in different weights.

A $k$--FCY $X$ is of \emph{strong} CY--type if it has only one $k$-Calabi--Yau structure located in the middle cohomology, and the natural map (for $2p=n-k$)
$$H^{n-p}(X,\Omega_X^p)\otimes H^1(X,TX)\lra H^{n-p+1}(X,\Omega_X^{p-1})$$
is an isomorphism.
\end{defin}

The notion of strong CY--type is the one which is in general required in the literature, as
in \cite{imcy}, where the case $k=3$ is investigated in a multitude of cases. However, we prefer here to consider the CY condition without the assumption on the deformation space.
In fact already in the case $k=2$ this assumption leaves out significant examples, such as the (Gushel--Mukai) index 2 Fano fourfold of genus 6.
Sticking to the examples relevant to this paper, $T_1(3,10)$ will be of strong K3 type, whereas $HI_i(3,10-i)$ (for $i=1,2$) will not satisfy this extra assumption. Finally, relevant examples of FK3 with multiple K3 structures include $T(2,10)$ or $P(1,10)$, while a Fano with mixed $(2,3)$--CY structure is $HT(2,9)$. Many other examples and computations can be found in \cite{fm}.

The main example of Fano varieties of Calabi-Yau type that will be treated in this paper is that
of hyperplane sections of Grassmannians.
We will show that hyperplane sections of Grassmannians $\Gr(k,V_n)$ carry a Hodge structure of (strong) Calabi-Yau type, extending in a weak form a result of Kuznetsov to the cases where $n$ and $k$ are not coprime.

\subsection{Cohomology of twisted forms on Grassmannians} \label{twisted}
The cohomology groups of sheaves of twisted differential forms on a Grassmannian
$\Gr=\Gr(k,V_n)$ have been extensively studied in \cite{snow}, who devised some combinatorial recipes to compute them. Let $\ell=n-k$. 
The basic observation is that the bundle of $j$-forms on $\Gr$ 
decomposes as 
$$\Omega^j_{\Gr} = \bigoplus_\alpha S_{\alpha^\vee}\kq^*\otimes S_\alpha T,$$
where the sum is over the set of all partitions $\alpha=(\alpha_1,\ldots, \alpha_k)$  
of size $\alpha_1+\cdots +\alpha_k=j$, such that $\ell\ge \alpha_1\ge \cdots\ge \alpha_k\ge 0$. 
Moreover, $\alpha^\vee$ is the dual partition, defined by $\alpha^\vee_m=\#\{r, \; \alpha_r\ge m\}$. 

The Borel--Bott--Weil theorem allows to decide whether such a partition $\alpha$ contributes to
the cohomology of $\Omega^j_{\Gr}(-i)$ (we will only need to consider the case where
$i>0$). The rule is the following. 
Denote by $A(i)$ the sequence $(\alpha_1-1+i,\ldots ,\alpha_k-k+i)$.    
Then $\alpha$ does contribute to the cohomology of $\Omega_{\Gr}^j(-i)$ if and only if 
the intersection of $A(i)$ with the interval $[-k,\ell-1]$ is contained in $A(0)$.

When this condition is fulfilled, observe that the largest integer 
of $A(i)$, that is $\alpha_1-1+i$, must be bigger or equal to $\ell$. Indeed, if it 
were not the case, then  $A(0)$ and $A(i)$ would both be contained in $[-k,\ell-1]$,
and then the condition would be that $A(i)\subset A(0)$, which is absurd.  So let $r$ be the
largest integer such that $\alpha_r-r+i\ge \ell$, and suppose that $r<k$. Then $\alpha_{r+1}-(r+1)+i$, being bigger than $-k$, must belong to $A(0)$:
there exists $s_1$ such that $\alpha_{r+1}-(r+1)+i=\alpha_{s_1}-s_1$ (and  then necessarily
$s_1\le r$). More generally, for any $t\ge 1$ such that $r+t\le k$, there must exist $s_t$ such that $\alpha_{r+t}-(r+t)+i=\alpha_{s_t}-s_t$. 

These strong combinatorial conditions can be nicely rephrased in terms 
of hook numbers \cite{snow}. When they are fulfilled, the partition $\alpha$ contributes
to exactly one twisted Hodge number $h^q(\Omega^j_{\Gr}(-i))$, and its contribution can
be computed as the dimension of a certain Schur power of $V_n$. 
  
\subsection{Hodge numbers of hyperplane sections}\label{HG}
Let $Y$ be a smooth  hyperplane section of $\Gr(k,n)$, of dimension $d=k(n-k)-1$.
By the Lefschetz hyperplane 
theorem, $Y$ has the same Hodge numbers as $\Gr(k,n)$ in degree smaller than $d$. So
the Euler characteristic of $\Omega^q_{Y}$ is 
$$\chi(\Omega^q_{Y})=(-1)^qh^{q,q}(\Gr(k,n))+(-1)^{d-q}h^{q,d-q}(Y),$$
for any $q\ne d-q$. Since we know the Hodge numbers of $\Gr(k,n)$, we just need to compute these
Euler characteristics in order to get all the Hodge numbers of $Y$. In order to
do so, we use the normal exact sequence and
its wedge powers, which yield the long exact sequences
$$0\lra\cO_Y(-q)\lra \Omega_{\Gr}(-q+1)_{|Y}\lra\cdots\lra \Omega^q_{\Gr |Y}\lra\Omega^q_{Y}\lra 0$$
for any $q>0$. Taking the alternate sum of the Euler characteristics, we get:

\begin{prop}\label{hodgenumbers}
The Hodge numbers of a smooth hyperplane section $Y$ of $\Gr=\Gr(k,V_n)$ can be 
computed as 
\begin{equation}
h^{q,d-q}(Y) = \sum_{i>0}(-1)^{d-q+i}\Big(
\chi (\Omega_{\Gr}^{q-i}(-i))-\chi(\Omega_{\Gr}^{q+1-i}(-i))\Big).
\end{equation}
\end{prop}

This formula can be implemented to compute the Hodge numbers effectively. 
Let us now turn to our main application. 

\smallskip
Kuznetsov proved in \cite[Corollary 4.4]{kuz-fractional} that when $k$ and $\ell$ are
coprime, and $d$ divides $n=k+\ell$, the derived category of a smooth hypersurface $Y$
of degree $d$ in the Grassmannian $\Gr(k,V_n)$ admits an exceptional collection whose
right orthogonal is a Calabi-Yau category. This implies that $Y$ is of pure derived Calabi-Yau type. When $k$ and $\ell$ are not coprime,
the  Grassmannian $\Gr(k,V_n)$ does not necessarily admit a rectangular Lefschetz decomposition 
and the situation is more complicated. We will prove the following much weaker statement,
but without any coprimality condition. 

\begin{theorem}\label{thm:sect-of-G}
Suppose that $n>3k$ and $k>2$. A smooth hyperplane section $Y$ of $\Gr(k,V_n)$ is of $N$ Calabi-Yau type for $N=k(n-k)+1-2n$.  
\end{theorem}

Note that the condition that $k>2$ is necessary, since a hyperplane section of $\Gr(2,V_n)$
has pure cohomology. Probably the condition that $n>3k$ can be improved, but note also that a 
hyperplane section of $\Gr(3,V_6)$ has pure cohomology. 

\proof 
Consider a partition $\alpha$, as in section \ref{twisted}, that contributes to the 
cohomology of $\Omega^j_{\Gr}(-i)$. Let $r$ be the
largest integer such that $\alpha_r-r+i\ge \ell$. As we observed, if $r<k$, there must
exist an integer $s=s_1\le r$ such that $\alpha_{r+1}-(r+1)+i=\alpha_{s}-s$.
From $i\ge \ell+r-\alpha_r$ and $i=\alpha_s-\alpha_{r+1}+r+1-s$ we deduce 
that $\alpha_s+\alpha_r\ge \ell+s-1$, and then  
$$i+j =\alpha_1+\cdots +2\alpha_s+\cdots +\alpha_r+\cdots +(r+1-s)>2\alpha_s+\alpha_r\ge  \frac{3}{2}(\alpha_s+\alpha_r)
\ge \frac{3\ell}{2}.$$

In the range $i+j\le \frac{3\ell}{2}$, the only partitions that contribute to the cohomology of
$\Omega_G^j(-i)$ must therefore be such that $\alpha_k-k+i\ge \ell$. Then their contribution
occurs in maximal degree, which means that 
$$\chi(\Omega_{\Gr}^j(-i))=(-1)^{\dim {\Gr}}h^{\dim {\Gr}}(\Omega_{\Gr}^j(-i))=(-1)^{\dim {\Gr}}h^0(\Omega_{\Gr}^{\dim {\Gr}-j}(i)).$$
The latter can then be deduced from the Borel-Weil theorem. To be more explicit, the partition 
$\alpha$ contributes by the dimension of the Schur power $S_{\hat{\alpha}}\CC^n$, 
where 
$$\hat{\alpha} = (\alpha_1+i-n,\ldots ,\alpha_k+i-n,-\alpha_\ell^\vee, \ldots -\alpha_1^\vee).$$ Finally, observe that the condition that
$\alpha_k-k+i\ge \ell$ implies that $i+j\ge n+\alpha_1+\cdots +\alpha_{k-1}$. We deduce 
that, for $n<\frac{3}{2}\ell$, or equivalently $\ell >2k$:
\begin{enumerate}
\item For $i+j<n$, $\chi(\Omega_{\Gr}^j(-i))=0$.  
\item For $i+j=n$, the only possibility is $\alpha = (0,\ldots ,0)$, hence $j=0$, $i=n$, and $\hat{\alpha}=(0,\ldots ,0)$; as a consequence, $\chi(\Omega_{\Gr}^j(-i))=\delta_{j,0}$.
\item For $i+j=n+1$, the only possibilities are $\alpha = (0,\ldots ,0)$, hence $j=0$, $i=n+1$
and $\hat{\alpha}=(1,\ldots ,1,0,\ldots ,0)$ (with $k$ ones); or $\alpha = (1,0,\ldots ,0)$, 
hence $j=1$, $i=n$ and $\hat{\alpha}=(1,0,\ldots ,0,-1)$; as a consequence, 
$$\chi(\Omega_{\Gr}^j(-i))=\delta_{j,0}\binom{n}{k}+\delta_{j,1}(n^2-1).$$
\end{enumerate}
Using Proposition \ref{hodgenumbers}, we deduce that 
$h^{q,d-q}(Y)=0$ for $q<n-1$, while $h^{n-1,d-n+1}(Y)=1$. This proves that $Y$ is of $N$ Calabi-Yau type. \qed 

\smallskip
Note that the next Hodge number is 
$$h^{n,d-n}(Y)= (-1)^d\Big(\chi(\cO_{\Gr}(-n))-\chi(\cO_{\Gr}(-n-1))+\chi(\Omega_{\Gr}(-n))\Big)=\binom{n}{k}-n^2,$$
which is equal to $h^1(Y,TY)$.  This suggests that $Y$ is of strong $N$ Calabi-Yau type, 
but we did not check it.

\section{Projections and Jumps}\label{sect:proj-jump}
In this section we introduce two geometric correspondences between Grassmannians. The first one
is a \linedef{projection}: given a linear projection $V_n\ra V_m$, there is for any $k$ an induced (rational) projection  from $\Gr(k,V_n)$ to $\Gr(k,V_m)$. The second one is a \linedef{jump}: it goes from $\Gr(k,V_n)$ to $\Gr(h,V_n)$ and is obtained by passing through the partial flag $\Fl(h,k,V_n)$. We will analyze how these correspondences restrict to subvarieties of the form $I_r(3,n)$ and their hyperplane sections $HI_r(3,n)$.

\subsection{Projections of Grassmannians}
Given $V_n$ and $V_m$ complex vector spaces of dimension $n$ and $m$, and $k< m < n$, let
$\pi: V_n \to V_m$ be a projection from a fixed $(n-m)$-dimensional vector subspace $V_{n-m} \subset V_n$. For a given  $k$-dimensional subspace $U \subset V_n$, the image $\pi(U) \subset V_m$ is $k$-dimensional if $U \cap V_{n-m} = 0$. Thus $\pi$ induces a rational surjective map $\pi: \Gr(k,V_n) \dashrightarrow \Gr(k,V_m)$ which we call a \linedef{projection}. We focus here on the most simple case, that is, $m=n-1$, so that 
$$\pi: \Gr(k,V_n) \dashrightarrow \Gr(k,V_{n-1})$$
is determined by the choice of a line $V_1 \subset V_n$.

If $U \subset V_{n-1}$ is a $k$-dimensional subspace, then the fiber of $\pi$ 
over $[U]$ in $\Gr(k,V_n)$ consists of those $k$-dimensional subspaces of $V_n$ of the form
$$U_\phi : = \{ u + \phi(u) \vert u \in U \}, \qquad \phi\in Hom(U,V_1).$$ In particular this fiber is an affine space of dimension $k$. Moreover, $\pi$ is not defined on the subset of $\Gr(k,V_n)$
whose elements are the $k$-dimensional subspaces of $V_n$ containing $V_1$. This subset is naturally isomorphic to $\Gr(k-1,V_{n-1})$, and we will resolve the indeterminacies of $\pi$ 
by blowing it up. We  end up with a diagram:
\begin{equation}\label{eq:projection-the-general-one}
\xymatrix{
E \ar@{^{(}->}[r] \ar[d]_p & X \ar[d]_\sigma \ar[dr]^\tau \\
\Gr(k-1,V_{n-1}) \ar@{^{(}->}[r] & \Gr(k,V_n) \ar@{-->}[r]^\pi & \Gr(k,V_{n-1}),
}
\end{equation}
where $\sigma$ is the blow-up of $\Gr(k,V_n)$ along $\Gr(k-1,V_{n-1})$ with exceptional divisor $E$. We claim that $\tau: X \to \Gr(k,V_{n-1})$
is the projective bundle $X \simeq \PP_{\Gr(k,V_{n-1})}(\ko \oplus Hom(\ku, V_1))$, with the map
$\sigma$ given by $$\sigma ([z,\phi])=Ker(zId_{V_1}-\phi)\subset V_1\oplus \ku\subset V_n.$$ 
Indeed, $\sigma$ as defined by this formula is birational outside the divisor  $$E= \PP_{\Gr(k,V_{n-1})}(Hom(\ku, V_1))=\PP_{\Gr(k,V_{n-1})}(\ku^*),$$ which 
is isomorphic to the flag variety $\mathrm{Fl}(k-1,k,V_{n-1})$. And the restriction 
of $\sigma$ to $E$ is the natural projection $p: E=\mathrm{Fl}(k-1,k,V_{n-1}) \to \Gr(k-1,V_{n-1})$, which is also the projective bundle  $\PP_{\Gr(k-1,V_{n-1})}(\kq)$.
This readily implies that $\sigma$ is the blow-up of $\Gr(k-1,V_{n-1})$ inside 
$\Gr(k,V_n)$, as claimed.

\smallskip
Now we would like to study the restriction of $\pi$ to varieties of the form $I_r\Gr(k,V_n)$, or, better, to their
hyperplane sections. Most relevant is the case $k=3$, where a hyperplane section $T(3,n)$ is 
defined by a $3$-form $\Omega$. For a choice of a decomposition $V_n = V_1\oplus V_{n-1}$, 
we can write $\Omega = \Omega' + \omega \wedge e_1^*$, for $\Omega'$ (resp. $\omega$) a $3$-form (resp. a $2$-form) on $V_{n-1}$, and $e_1^*$ a linear form with kernel  $V_ {n-1}$. In this case
we will have to consider the subvariety $I(3,n-1)$ in $\Gr(3,V_{n-1})$ defined by $\omega$,
and its hyperplane section $HI(3,n_1)$  defined by $\Omega'$.

\subsection{Relating hyperplane sections of symplectic Grassmannians of $3$-planes.}
Let $HI_r(3,n)$ be a general hyperplane 
section, defined by a 3-form $\Omega$ on $V_n$,  of a $r$-th symplectic Grassmannian $I_r(3,n)$ defined by $2$-forms $\omega_1,\ldots,\omega_r$. 

As above, let us fix a decomposition $V_n=V_1 \oplus V_{n-1}$, and let us write 
$\Omega = \Omega' + \omega \wedge e_1^*$, for $\Omega'$ a $3$-form,  $\omega$ a $2$-form on $V_{n-1}$, and $e_1^*$ a generator of $V_{n-1}^\perp$. 
The forms $\omega_i$ restrict to $2$-forms on $V_{n-1}$, that we denote in the same way. 
Then, we can consider the $r$-th (resp. $(r+1)$-th) symplectic Grassmannian $I_r(3,n-1)$ (resp. $I_{r+1}(3,n-1)$) defined by the forms $\omega_i$ (resp. $\omega_i$ and $\omega$), and the hyperplane section $HI_{r+1}(3,n-1)$ of the latter, defined by the $3$-form $\Omega'$.

In general, the image of $I_r(3,n)$ by $\pi$ will not be contained in $I_r(3,n-1)$. In order 
to ensure this, we need to assume that each $\omega_i$ is singular, with kernel 

containing $V_1$.
We will in fact assume that 
\begin{equation}\label{eq:the-kernel-1-dimensional}
V_1 = \bigcap_{i=1}^r \mathrm{ker}(\omega_i) \quad \text{ is one-dimensional.}
\end{equation}
Condition \eqref{eq:the-kernel-1-dimensional} implies that the  $r$-tuple of 
forms $\omega_1,\ldots ,\omega_r$ is non generic,
unless $r=1$ and $n$ is odd.
In particular under this condition $I_r(3,n)$ can (and will in general) be singular,
and it can even be of bigger dimension than expected. One can have a 
partial control of these phenomena for small values of $r$, but in this paper we will only
consider in detail examples with $r=1$ and $n$ odd, so we do not push further the analysis of singularities and expected dimensions. We keep anyway considering projections for general values of $r$-tuples,
satisfying the above condition \eqref{eq:the-kernel-1-dimensional}.
(Alternatively, we could consider only the closure of the set of isotropic
$3$-planes that do not contain $V_1$. This will be irreducible of the correct dimension.)

Now consider the restriction $\pi': HI_r(3,n) \dashrightarrow I_r(3,n-1)$. Its fibers are the intersections of $HI_r(3,n)$ with the fibers of $\pi: \Gr(3,V_n) \dashrightarrow \Gr(3,V_{n-1})$. Recall that the fiber of $\pi$ over $U\in \Gr(3,V_{n-1})$ consists of the subspaces of $V_n$ of the form $U_\phi=\{u+\phi(u), \; u\in U\}$, for $\phi\in Hom(U,V_1)$. 
Identify the latter with $U^*$ by choosing for basis of $V_1$ the vector $e_1$ such that 
$\langle e_1^*,e_1\rangle =1$.
Such a $U_\phi$ then belongs to $HI_r(3,n)$ if and only if $U$ belongs to $I_r(3,n-1)$ and $\Omega'+\phi\wedge\omega=0$ on $U$. 
We shall therefore consider the subvariety $\widetilde{HI_r}(3,n)\subset \PP_{I_r(3,n-1)}(\ko\oplus\ku^*)$ parameterizing those points $[z,\phi]\in \PP(\ko\oplus \ku^*)$,
where $U$ belongs to $I_r(3,n-1)$, such that $z\Omega'+\phi\wedge\omega=0$ on $U$.
This defines a two-dimensional projective space in general, and a 
$3$-dimensionl projective space exactly when the condition is empty, that is, when $\Omega'$ and $\omega$ both vanish identically on $U$; in other words, when $U$ belongs to the hyperplane section $HI_{r+1}(3,n-1)$ of $I_{r+1}(3,n-1)$.

The map $\pi'$ is not defined on $Z'_r := Z_r \cap HI_r(3,n)$, which is isomorphic to the symplectic Grassmannian $I_{r+1} \Gr(2,V_{n-1})$ defined
by the $r+1$ forms $\omega_1, \ldots, \omega_r$ and $\omega$. In particular, $Z'_r$ is smooth
when these forms are general.

We end up with the following commutative diagram:
\begin{equation}\label{eq:projection-for-n-even}
\xymatrix{
E_r \ar@{^{(}->}[r] \ar[d]_p & \widetilde{HI_r}(3,n) \ar[d]_{\sigma} \ar[dr]^\tau&  F_r \ar@{_{(}->}[l]_j \ar[dr]^q \\
Z'_r \ar@{^{(}->}[r] & HI_r(3,n) \ar@{-->}[r]^-{\pi'} & I_r(3,n-1) & \ar@{_{(}->}[l]_-\iota HI_{r+1}(3,n-1),
}
\end{equation}
where $\sigma$ is birational over $HI_r(3,n)$ and an isomorphism outside $Z'_r$, and $p$ is the
restriction of the exceptional divisor $E_r \to Z'_r$. Moreover, $F_r$ is the locus $\tau^{-1} HI_{r+1}(3,n-1)$,  which has codimension $3$ in $\widetilde{HI_r}(3,n)$. Finally $q$ is the restriction of $\tau$ to $F_r$, whose fibers are $\PP^3$'s, while the other fibers of $\tau$
are $\PP^2$'s. Recall that $\LL$ denotes the class of the affine line in the Grothendieck ring $K_0(\mathrm{Var}(\CC))$ of complex algebraic varieties. We deduce:

\begin{prop}\label{prop:K0-projection}
In the Grothendieck ring $K_0(\mathrm{Var}(\CC))$, the following relations hold:
\begin{equation}\label{eq:rel-in-K0-proj}
[HI_r(3,n)] - [HI_{r+1}(3,n-1)]\LL^3 = [I_r(3,n-1)][\PP^2] - [I_{r+1} \Gr(2,n-1)][\PP^{c-2}]\LL
\end{equation}
\end{prop}

\begin{proof}
By the above description, the class of $[\widetilde{HI_r}(3,n)]$ in $K_0(\mathrm{Var}(\CC))$ can be written as
$$[\widetilde{HI_r}(3,n)]=[HI_r(3,n)]+ [Z'_r][\PP^{c-2}]\LL$$
by decomposing $\sigma$ into an isomorphism outside $Z'_r$ and the projective bundle $p$, and as
$$[\widetilde{HI_r}(3,n)]=[I_r(3,n-1)][\PP^2] + [HI_{r+1}(3,n-1)]\LL^3$$
via the map $\tau$. The conclusion follows by comparison.
\end{proof}

When the varieties involved in \eqref{eq:projection-for-n-even} are smooth, $\sigma$
is just the blow-up of $Z'_r$ and we can enhance the previous relation at the level of
derived categories. This happens only for 
\begin{equation}\label{eq:smoothness-assumpt}
\begin{array}{l}
n \text{ is odd and } r \leq 1, \text{ or }\\
n \text{ is even and } r=0.
\end{array}
\end{equation}
\begin{prop}\label{prop:deco-projection}
Assume \eqref{eq:smoothness-assumpt} holds, and denote by $c$ the codimension of $Z'_r$ in $HI_r(3,n)$.
There are fully faithful functors 
$$\begin{array}{rl}
\Phi: \Db(HI_{r+1}(3,n-1)) &\longrightarrow \Db(\widetilde{HI_r}(3,n)),\\
\Psi_i:\Db(Z'_r) &\longrightarrow \Db(\widetilde{HI_r}(3,n)),
\end{array}$$
for any  integer $i$, and semiorthogonal decompositions of $\Db(\widetilde{HI_r}(3,n))$ as:
\begin{equation}\label{eq:our-deco-proj}
\sod{\Phi \Db(HI_{r+1}(3,n-1)),\tau^* \Db(I_r(3,n-1)), \ldots, \tau^* \Db(I_r(3,n-1)) \otimes \ko(2H)},
\end{equation}
\begin{equation}\label{eq:Orlov-deco-proj}
\sod{\Psi_1 \Db(Z'_r),\ldots \Psi_{c-1} \Db(Z'_r), \sigma^* \Db(HI_r(3,n))}.
\end{equation}
\end{prop}

\begin{proof}
The semiorthogonal decomposition \eqref{eq:our-deco-proj} is obtained as a particular case of Proposition \ref{prop:main-sod}, Corollary \ref{cor:ourcases}, since the codimension of $F_r$ is $3$ and the general fiber of $\tau$ is a $2$-dimensional. The calculation of the normal bundle is the same as in Lemma \ref{lem:norm-proj}. The semiorthogonal decomposition \eqref{eq:Orlov-deco-proj} is Orlov's decomposition
for a blow-up \cite{orlovprojbund}.
\end{proof}

\begin{prop}\label{prop:Hodge-projection}
Assume \eqref{eq:smoothness-assumpt} holds, and denote by $c$ the codimension of $Z'_r$ in $HI_r(3,n)$. There are isomorphisms of integral Hodge structures
\begin{equation}\label{eq:our-hodge-proj}
H^j(\widetilde{HI_r}(3,n),\CC)  = H^{j-6}(HI_{r+1}(3,n-1))(-3) \oplus \bigoplus_{i=0}^2 H^{j-2i}(I_r(3,n-1))(-i),
\end{equation}
\begin{equation}\label{eq:blowup-hodge-proj}
H^j(\widetilde{HI_r}(3,n),\CC) = H^j(HI_r(3,n),\CC) \oplus \bigoplus_{i=1}^{c-1} H^{j-2i}(Z'_r,\CC)(-i).
\end{equation}
\end{prop}

\begin{proof}
The Hodge decomposition \eqref{eq:our-hodge-proj} is a special case of Proposition \ref{prop:main-Hdg-deco}. The Hodge decomposition \eqref{eq:blowup-hodge-proj} follows 
from the well-known formula for  blow-ups (see, e.g., \cite[7.7.3]{voisin-book}). 
\end{proof}

Notice that the Hodge numbers $h^{p,q}(\widetilde{HI_r}(3,n))$ can also be computed from
Proposition \ref{prop:K0-projection} via the Hodge motivic evaluation.

\subsection{Jumps and hyperplane sections}

Let $h<k$ be integers in $\{1,\ldots,n-1\}$. 
Consider the flag variety $\Fl(h,k,V_n)$ with its projections $p$ to $\Gr(h,V_n)$ and $q$ 
to $\Gr(k,V_n)$. The fibers of $q$ are Grassmannians $\Gr(h,k)$: given a $U \subset V_n$ of dimension $k$, the fiber over $U$ is the Grassmannian $\Gr(h,U)$. The fibers of
$p$ are Grassmannians $\Gr(n-k,n-h)$: given a $W \subset V_n$ of dimension $h$, the fiber of $W$ is the Grassmannian $\Gr(n-k,V_n/U)$. The correspondence $p_*q^*$ (on cohomology, derived categories etc.)
will be called an \linedef{$(h,k)$-jump on $V_n$}. We denote by $\ko(H)$ and $\ko(L)$ the Pl\"ucker
relative line bundles of the Grassmannian fibrations $p$ and $q$ respectively.

We will describe in details only the simplest case, where $h=k-1$, and the induced correspondence
on subvarieties of $\Gr(k,V_n)$. So consider 
the flag variety $\Fl(k-1,k,V_n)$ with its projections $p$ to $\Gr(k-1,V_n)$ and $q$ 
to $\Gr(k,V_n)$. The fibers of $p$ are projective spaces of dimension $n-k$, those of
$q$ are projective spaces of dimension $k-1$.

\medskip

First of all, consider a hyperplane section $Y$ of $\Gr(k,V_n)$. Such a $Y$ is defined 
by a $k$-form $\Omega$ on $V_n$, and we let $q^* Y \subset \Fl(k-1,k,V_n)$ be defined by $q^*\Omega$.
Then $q : q^* Y \to Y$ is a $\PP^{k-1}$-bundle. We want to understand the restriction of $p$ to $q^*Y$. 
Let $U=\langle u_1,\ldots, u_{k-1}\rangle \subset V_n$, be a point in $\Gr(k-1,V_n)$. The fiber of $p$ over $U$ is naturally identified with $\PP(V_n/U)$. Points in such a fiber that belong to $q^*Y$ are identified with the linear subspace of $\PP(V_n/U)$
defined by the linear form $\Omega (u_1,\ldots, u_{k-1}, -)$. This subspace is a hyperplane,
except when $U$ belongs to the locus $Z$ where this form vanishes, in which case the whole
fiber of $p$ over $U$ is contained in $q^*Y$. Note that $Z$ is the zero locus of the section
of $\kq^*(1)$ defined by $\Omega$, so it is in general smooth of codimension $n-k+1$. 
So the $(k-1,k)$-jump on $V_n$ induces the following diagram:
$$\xymatrix{
 &  q^*Y \ar@{->}[ld]_{p}\ar@{->}[rd]^{q}  &  \\
 Z\subset \Gr(k-1,V_n) & & Y 
}$$
where $q: q^*Y \to Y$ is a $\PP^{k-1}$-bundle with relative ample line bundle $\ko(L)$, and 
$p: q^* Y \to \Gr(k-1,V_n)$ is a
$\PP^{n-k-1}$-bundle over $\Gr(k-1,V_n)\setminus Z$ and a $\PP^{n-k}$-bundle 
over $Z$ with relative ample line bundle $\ko(H)$. Let $c$ be the codimension of $Z$ in $\Gr(k-1,V_n)$.
We deduce the following Propositions.

\begin{prop}\label{prop:K0-jump}
The following relation holds in the Grothendieck ring $K_0(\mathrm{Var}(\CC))$:
\begin{equation}\label{eq:rel-in-K0-jump}
[Y][\PP^{k-1}]-[Z]\LL^{n-k}=[\Gr(k-1,V_n)][\PP^{n-k-1}].
\end{equation}
\end{prop}

\begin{proof}
By the above description, the class of $[q^*Y]$ in $K_0(\mathrm{Var}(\CC))$ can be written
as
\begin{equation}\label{eq:jumping-diagram-Hsection}
[q^*Y]=[Y][\PP^{k-1}]
\end{equation}
by the projective bundle formula, and as
$$[q^*Y]=[\Gr(k-1,V_n)][\PP^{n-k-1}] + [Z]\LL^{n-k}$$
via the map $p$. The proof follows by comparison.
\end{proof}

Note that we can rewrite this relation as 
$$[Z]\LL^{n-k}=
([\Gr(k,V_n)][\PP^{k-1}]-[\Gr(k-1,V_n)][\PP^{n-k-1}])+([Y]-[\Gr(k,V_n)])[\PP^{k-1}].$$
As far as Hodge numbers are concerned, by the Lefschetz hyperplane theorem the difference 
$[Y]-[\Gr(k,V_n)]$ will
not contribute in degree smaller than the dimension of $Y$. So up to degree $d_0=\dim Y-2(n-k)$, 
the Hodge numbers of $Z$ will be determined by the class $C=[\Gr(k,V_n)][\PP^{k-1}]-[\Gr(k-1,V_n)][\PP^{n-k-1}]$.
This is a polynomial in $\LL$ that we can compute as follows. Remember that the class of the Grassmannian $\Gr(k,V_n)$
is given by the $\LL$-binomial polynomial:
$$[\Gr(k,V_n)]=\frac{(1-\LL)(1-\LL^2)\cdots (1-\LL^n)}{(1-\LL)\cdots (1-\LL^k)(1-\LL)\cdots (1-\LL^{n-k})}.$$
Observe that the class of the flag variety $\mathrm{Fl}(k-1,k,n)$ can be computed using either one of its two natural 
projections to Grassmannians. We get:
$$[\mathrm{Fl}(k-1,k,n)]=[\Gr(k,V_n)][\PP^{k-1}]=[\Gr(k-1,V_n)][\PP^{n-k}].$$
This implies that $C=[\Gr(k-1,V_n)]\LL^{n-k}$. Since $d_0=\dim Z-(k-1)$, we deduce:

\begin{cor}\label{hodgeZ}
Up to degree $\dim Z-k$, the variety $Z$ has the same Hodge numbers as the ambient Grassmannian
$\Gr(k-1,V_n)$. In particular it has only pure cohomology in this range, and its Picard number is one
as soon as $\dim Z\ge k+2$.

Moreover the non pure cohomology of $Z$ appears in degree $\dim Z-k-1+2m$, for $1\le m\le k$, 
and in each of these degrees it is isomorphic to the non pure cohomology of $Y$. 
\end{cor}

A different argument can be used to establish the slightly more precise result that the restriction morphism $H^m(\Gr(k-1,V_n),\ZZ)\lra H^m(Z,\ZZ)$ is an isomorphism
in degree $m\le \dim Z -k$: we can use the  Barth-Lefschetz type theorems proved
by Sommese for subvarieties with $p$-ample normal bundle \cite[Proposition 2.6]{sommese}. Indeed we 
claim that $Z$ has $(k-1)$-ample normal bundle. In fact this normal bundle is the restriction of 
$\mathcal{Q}^*(1)$, whose bundle of hyperplanes is the flag variety $\mathrm{Fl}(k-1,k,n)$. Moreover the morphism 
defined by the relative hyperplane bundle is the projection to $\Gr(k,n)$. Since the fibers of 
this projection have dimension $(k-1)$, the bundle $\mathcal{Q}^*(1)$ is $(k-1)$-ample by definition. 
\smallskip

Let us now turn to derived categories:

\begin{prop}\label{prop:semiorth-for-Hsections-in-jumps}
There is a semiorthogonal decomposition:
\begin{equation}\label{eq:orlov-deco-jump}
\Db(q^* Y) = \sod{q^*\Db(Y),\ldots,q^*\Db(Y)\otimes \ko((k-1)L)}.
\end{equation}

If moreover the codimension of $Z$ satisfies $c \geq n-k-1$, and $Z$ is smooth, there is a fully faithful functor
$\Phi: \Db(Z) \to \Db(q^* Y)$ and a semiorthogonal decomposition:
\begin{equation}\label{eq:our-deco-jump}
\Db(q^* Y) = \sod{\Phi\Db(Z), p^*\Db(\Gr(k-1,V_n)),\ldots,p^*\Db(\Gr(k-1,V_n))\otimes \ko((n-k-2)H)}.
\end{equation}
\end{prop}

\begin{proof}
The semiorthogonal decomposition \eqref{eq:orlov-deco-jump} is just Orlov's decomposition for projective
bundles \cite{orlovprojbund}. The semiorthogonal decomposition \eqref{eq:our-deco-jump} is a special
case of Proposition \ref{prop:main-sod}, since the general fiber of $p$ is $\PP^{n-k-2}$ and the
locus $p^{-1}Z$ has codimension $c-1$ in $q^*Y$. In particular, it is a special case of
Corollary \ref{cor:ourcases}, the calculation of the normal bundle is the same as in Lemma \ref{lem:norm-FL23}.
\end{proof}

Recall that for $k,n$ coprime, the derived category of $Y$ decomposes into a $N$-CY full 
subcategory $\cat{A}_Y$, with $N=\dim Y-(2n-2)$, and a bunch of exceptional objects. From 
the previous decompositions, we can expect $\Db(Z)$ to decompose into $k$ copies of $\cat{A}_Y$, 
and ${n-1\choose k-2}$ exceptional objects. This suggests that there could exist a rectangular Lefschetz
decomposition when $k$ divides the binomial coefficient ${n-1\choose k-2}$. If $k$ is a prime
number, this is equivalent to $n\ne 0,-1$ mod $k$. 

\smallskip Finally we can compare Hodge structures: 

\begin{prop}\label{prop:Hodge-jump}
There is an isomorphism of integral Hodge structures
\begin{equation}\label{eq:proj-hodge-jump}
H^j(q^*Y,\CC)  = \bigoplus_{i=0}^{k-1} H^{j-2i}(Y)(-i).
\end{equation}

There is an isomorphism of integral Hodge structures
\begin{equation}\label{eq:our-hodge-jump}
H^j(q^*Y,\CC) = H^{j-2t}(Z,\CC)(-t) \oplus \bigoplus_{i=0}^{n-k-2} H^{j-2i}(\Gr(k-1,V_n),\CC)(-i),
\end{equation}
where $t=n-k-1$.
\end{prop}

\begin{proof}
The Hodge decomposition \eqref{eq:our-hodge-jump} is a special case of Proposition \ref{prop:main-Hdg-deco}. The Hodge decomposition \eqref{eq:proj-hodge-jump} is the well-known formula for the projective bundle. Notice that a computation of the dimensions $h^{p,q}(q^*Y)$ can be also obtained as corollary of Proposition \ref{prop:K0-jump} via the Hodge motivic evaluation.
\end{proof}

\subsection{Jumping from hyperplane sections of $\Gr(3,V_n)$, to congruences of lines and further}

Here we detail two special cases of the above construction, namely the $(2,3)$-jump and the $(1,2)$-jump on $V_n$, and the induced correspondences on a general hyperplane section $T(3,n)$ of $\Gr(3,V_n)$. We are then in the above case with $k=3$, so that $T(3,n)$ is our notation for the hyperplane section, and $T(2,n)$ is our notation for $Z$. In the diagram \eqref{eq:jumping-diagram-Hsection} the map $q$ is a $\PP^2$-bundle and the map $p$ is generically a $\PP^{n-4}$-bundle, and a $\PP^{n-3}$ bundle over $T(2,n)=Z$.

If we denote by $\Omega$ the $3$-form on $V_n$ defining the hyperplane section $T(3,n)$, the congruence $T(2,n) \subset \Gr(2,V_n)$ is the locus of planes $U=\langle 
u_1, u_2\rangle$ such that $\Omega(u_1,u_2,-)$ is the trivial linear form on $V_n$. In other 
words, $T(2,n)$ is the zero-locus of the section of $\kq^*(1)$ defined by $\Omega$. If the latter
is general, this implies that $T(2,n)$ is smooth of dimension $n-2$, with canonical bundle 
$\cO_{T(2,n)}(-3)$. These congruences of lines have been studied in \cite{dpfmr}.

\medskip
Notice that for $U$ in $T(2,n)$, and for any $u$ in $U$, the two-form $\Omega(u,-,-)$ on $V_n$
descends to a two-form $\bar\Omega_u$ on $Q=V_n/U$. We can give a precise characterization of the smoothness of $T(2,n)$ at $U$ in terms of this pencil of two-forms on $Q$. 

\begin{lemma} 
$T(2,n)$ is singular at $U$ if and only if the two-forms $\bar\Omega_u$ on $Q$ have
a common line in their kernel.
\end{lemma} 

\proof 
$T(2,n)$ is singular at $U$ exactly when the morphism $T_UG(2,n)\lra Q^*(1)$ defined
by $\Omega$ is not surjective. Dualizing, we get the map from $Q\otimes\wedge^2U$ 
to $Hom(Q,U)$ defined by 
$$q\otimes u_1\wedge u_2 \mapsto \Omega(q,u_1,-)u_2-\Omega(q,u_2,-)u_1.$$
The right hand side vanishes, for $u_1, u_2$ a basis of $U$, when $q$ belongs to the 
kernel of the two-forms $\bar\Omega_{u_1}$ and $\bar\Omega_{u_2}$.\qed

\medskip
Now let us consider the next case, that is the $(1,2)$-jump on $V_n$. In this case, we have the flag variety $\mathrm{Fl}(1,2,V_n)$ and the maps $p$ to $\Gr(1,V_n)\simeq \PP^{n-1}$, which is a $\PP^{n-2}$-bundle, and $q$ to $\Gr(2,V_n)$, which is a $\PP^1$-bundle. Consider the variety $T(2,n)$ and its preimage $q: q^*T(2,n) \to T(2,n)$ inside $\mathrm{Fl}(1,2,V_n)$, which is a $\PP^1$-bundle. Now restrict the map $p$ to $q^*T(2,n)$, to get a map $p: q^*T(2,n) \to \PP^{n-1}$. A line $L = \sod{l} \subset V_n$ is in the image of $p$ if and only if the form $\Omega(l,-,-)$ is degenerate as a form on $V_n/L$. In particular, we can distinguish two cases:

\begin{itemize}
\item If $n$ is even, every line sits in the image of $p$, and the projection $p: q^*T(2,n) \to \PP^{n-1}$ is birational. The exceptional locus is $P(1,n) \subset \PP^{n-1}$ and has codimension $3$. For $\Omega$ general, its singular locus is the set of lines $L = \sod{l}$ such that 
the form $\Omega(l,-,-)$ has corank at least five, and this locus has codimension ten; in particular $P(1,n)$ is smooth only for $n\le 10$. In this case $p$ is just the blow-up of $\PP^{n-1}$ along $P(1,n)$.
\item If $n$ is odd, the image of the projection $p: q^*Z \to \PP^{n-1}$ is the Pfaffian 
hypersurface $P(1,n) \subset \PP^{n-1}$ and $p$ is generically a $\PP^1$-bundle. 
For $\Omega$ general, the singular locus $S \subset P(1,n)$ has codimension $5$, so
that $P(1,n)$ is smooth for $n \leq 5$, and $p$ is a $\PP^3$-bundle over the smooth
locus of $S$. Moreover $S$ is smooth for $n\le 15$.
\end{itemize}

\begin{prop}\label{prop:jump-T2-to-T3}
We have the following relations in the Grothendieck group $K_0(\mathrm{Var}(\CC))$:
\begin{itemize}

\item For any $n$:
\begin{equation}\label{eq:jump-T2-to-T3}
[T(3,n)][\PP^2]=[\Gr(2,V_n)][\PP^{n-4}]+[T(2,n)]\LL^{n-3}
\end{equation}

\item If $n\le 10$ is even:
\begin{equation}\label{eq:jumping-to-peskine-even}
[T(2,n)][\PP^1]=[\PP^{n-1}]+[\PP^1][P(1,n)]\LL.
\end{equation}

\item If $n\le 15$ is odd:
\begin{equation}\label{eq:jumping-to-peskine-odd}
[T(2,n)][\PP^1]=[\PP^1]([P(1,n)]+[S]\LL^2).
\end{equation}
\end{itemize}
\end{prop}

As before, there are also versions of this statement for derived categories and Hodge structures: 

\begin{prop}\label{prop:decoforpeski}
Assume that $T(2,n)$ is smooth. 
There is a semiorthogonal decomposition:
\begin{equation}
\Db(q^*T(2,n))=\sod{q^* \Db(T(2,n)),q^*\Db(T(2,n))(L)},
\end{equation}
where $L$ is the relative ample line bundle of the map $q$.
If $n\le 10$ is even, and $P(1,n)$ is smooth, there are fully faithful functors $\Phi_i:\Db(P(1,n)) \to \Db(q^*T(2,n))$ for any $i \in \ZZ$ and a semiorthogonal decomposition
\begin{equation}
\Db(q^*T(2,n))=\sod{\Db(\PP^{n-1}),\Phi_1\Db(P(1,n)),\Phi_2\Db(P(1,n))}.
\end{equation}

If $n\le 5$ is odd and $P(1,n)$ is smooth, there is a semiorthogonal decomposition
\begin{equation}
\Db(q^*T(2,n))=\sod{p^*\Db(P(1,n)),p^*\Db(P(1,n))(H)},
\end{equation}
where $H$ is the relative ample line bundle of the map $p$.
\end{prop}

\begin{prop}\label{prop:peski-Hodge}
Assume that $T(2,n)$ is smooth.
There is an isomorphism of integral Hodge structures:
\begin{equation}
H^j(q^*T(2,n),\CC)= H^j(T(2,n),\CC) \oplus H^{j-2}(T(2,n),\CC)(-1).
\end{equation}

If $n\le 10$ is even, and $P(1,n)$ is smooth, there is an isomorphism of integral Hodge structures
\begin{equation}
H^j(q^*T(2,n),\CC)=H^j(\PP^{n-1},\CC) \oplus H^{j-2}(P(1,n),\CC)(-1) \oplus H^{j-4}(P(1,n),\CC)(-2).
\end{equation}

If $n\le 5$ is odd and $P(1,n)$ is smooth, there is a an isomorphism of Hodge structures
\begin{equation}
H^j(q^*T(2,n),\CC)=H^j(P(1,n),\CC) \oplus H^{j-2}(P(1,n),\CC)(-1).
\end{equation}
\end{prop}

\subsection{The index of $T(2,n)$}

In \cite{dpfmr}, Problem, section 4.4, the authors ask about the Hodge numbers of $T(2,n)$.
Proposition \ref{prop:jump-T2-to-T3} allows to deduce them from the Hodge numbers of $T(3,n)$.
Morever, since $T(3,n)$ is just a hyperplane section, the Hodge numbers of $T(3,n)$ are given
by Proposition \ref{hodgenumbers}. In fact Corollary \ref{hodgeZ} gives almost all the Hodge 
numbers of $T(2,n)$ quite directly. 
In particular $T(2,n)$ has Picard number one as soon as $n\ge 7$ (and note that $T(2,6)\simeq
\PP^2\times\PP^2$). 

\begin{prop} $T(2,n)$ has index $3$. \end{prop}  

\proof By adjunction, the canonical line bundle of $T(2,n)$ is the restriction of 
$\ko(-3)$, and we have to show that the restriction of the Pl\"ucker line 
bundle to $T(2,n)$ is not divisible. First observe that if $h$ is $m$-divisible, 
then the degree of $T(2,n)$ in the Pl\"ucker embedding must be divisible by  
$m^{n-2}$. This degree can be computed explicitly as follows. The fundamental class
of $T(2,n)$ in the Chow ring of the Grassmannian is 
$$[T(2,n)]=c_{n-2}(\kq^*(1))=\sigma_{1,1}\sum_{i\ge 1}h^{n-2i-3}\sigma_{2i-1}+\delta_{n\;even}\sigma_{n-2},$$
where $h$ is the hyperplane class and we use standard notations for the Schubert 
cycles $\sigma_k$ and $\sigma_{1,1}$. Using the Frame-Robinson-Thrall formula and
Corollary $3.2.14$ of \cite{schubert}, we deduce that 
$$\deg T(2,n) = \sum_{i\ge 1}\frac{2i}{n-2} {2n-2i-5\choose n-2i-2}+\delta_{n\;even}.$$ 
Moreover the terms in the summation above
decrease when $i$ gets bigger, and since there are at most $(n-2)/2$ terms we deduce that 
$\deg T(2,n)\le {2n-7\choose n-4}\le 2^{2n-7}$. So we just need to check that the hyperplane class is not divisible by $2$ or by $3$. We use the following trick. It is a straightforward exercise in Schubert calculus to check that:

\begin{lemma} Let $\epsilon_n=0$ for $n$ even, $\epsilon_n=1$ for $n$ odd. Then 
$$a_n := \int_{T(2,n)}h\sigma_{n-3}=\frac{n+\epsilon_n-4}{2}, \qquad b_n:=\int_{T(2,n)}h^2\sigma_{n-4}=\frac{n^2-\epsilon_n-12}{4}.$$
\end{lemma}

For $n=2p$, $b_{n}=p^2-3$ is never divisible neither by $4$ nor by $9$, so $h$ is neither
$2$-divisible nor $3$-divisible. For $n=2p+1$, $b_{n}=p^2+p-3$ is always odd, so $h$ is not
$2$-divisible; moreover $b_n$ is divisible by $9$ if and only if $p=3$ or $p=5$ mod $9$, 
and then $a_n=p-1$ is not divisible by $3$, so $h$ is not $3$-divisible. This concludes 
the proof. \qed

\section{The nested construction for the Debarre-Voisin hypersurface}\label{sect:DV}

In this section, we focus on a very special case, the hyperplane section $Y:=T(3,10)$ of
the Grassmannian $\Gr(3,V_{10})$. 

\subsection{A cascade of projections} 
This hypersurface $Y$ was considered in \cite{debarre-voisin}, where it is proved that the copies 
of $\Gr(3,6)$ that it contains (and their degenerations) are parametrized by a hyperK\"ahler fourfold. 
This is reflected in the fact that $Y$ is both of strong 
K3-type (as recalled in Theorem \ref{thm:sect-of-G}) and of pure derived K3 type. Indeed, 
\begin{equation}\label{eq:sod-for-DV}
\Db(Y)= \sod{\cat{A},E_1,\ldots,E_{108}},
\end{equation}
where $\cat{A}$ is a K3 category and the $E_i$'s are exceptional objects \cite{kuz-fractional}.

The vanishing cohomology $H^{p,q}_{van}(Y)$
has the following dimensions \cite{debarre-voisin}:
\begin{equation}\label{eq:Hodge-for-DV}
h_{van}^{10-p,10+p}(Y) = \begin{cases}
1 \quad \text{ if } p=\pm 1, \\
20 \quad \text{ if } p=0,\\
0 \quad \text{ otherwise.}
\end{cases}
\end{equation}
Moreover, if $Y$ is very general, the Hodge structure on the vanishing cohomology $H^{20}_{van}(Y,\CC)$ is a simple weight two Hodge structure \cite[Thm. 2.2]{debarre-voisin}, 
and is therefore the minimal indecomposable subHodge structure containing $H^{9,11}(Y)$.

\begin{defin}
Let $K \subset H^{20}(Y,\CC)$ denote the minimal indecomposable sub-Hodge structure containing $H^{9,11}(Y)$.
\end{defin}

It is not known if $K$ coincides with $H^{20}_{van}(Y,\CC)$ in general. We can wonder whether a similar
phenomenon can be traced on the noncommutative side. Indeed, one would expect that the category $\cat{A}$ appearing in (\ref{eq:sod-for-DV})
is in general not the derived category of a K3 surface but rather 
a deformation of it, and we can state the following folklore conjecture.

\begin{conjecture}\label{conj:hodgeandD}
If $Y \subset \Gr(3,V_{10})$ is a very general hyperplane section,
there is no smooth and projective $K3$ surface $W$ and no Brauer class $\alpha$ on $W$ such
that $\cat{A} \simeq \Db(W,\alpha)$.
\end{conjecture}

In any case, both the category $\cat{A}$ and the Hodge structure $K$ are relevant objects to study.
For example, one can wonder about a categorical Torelli theorem, by asking to which extent
the category $\cat{A}$ determines the isomorphism class of $Y$, mimicking the case of
cubic fourfolds (\cite{huybrechts-rennemo,BLMS,Li-Per-Xiao}). Notably, the birational
counterpart is certainly not true since $Y$ is rational (it is birational to $\Gr(3,V_9) \times \PP^2$,
see diagram \eqref{eq:the-nested-diagram}). Indeed $Y$
is twenty-dimensional, while $\cat{A}$ should be realized in varieties
of dimension $6$ such as the Peskine variety (see conjecture \ref{conj:decompositions}), so that
it is not surprising that $\cat{A}$ is not an obstruction to rationality in this case.
Other very interesting questions on $\cat{A}$ and $K$ are related to the construction of hyperk\"ahler moduli
of subvarieties of $Y$ (see \cite{debarre-voisin}) as moduli spaces of objects in $\cat{A}$.

\medskip
We will apply the correspondences described in Section \ref{sect:proj-jump}, to show that several Fano varieties of $K3$ type can be geometrically related to $Y$ in such a way that $K$ is invariant under these correspondences. Moreover there are strong evidences for $\cat{A}$ to be invariant as well.

We use the following notation:

\begin{itemize}
\item[] $Y= T(3,10)$ the hyperplane section of $\Gr(3,V_{10})$, of dimension $20$.
\item[] $Z \subset Y$ the exceptional locus of a general projection $\pi':Y \dashrightarrow \Gr(3,9)$. Then $Z \simeq I(2,9)$, of codimension $7$ in $Y$.
\item[] $Y_1=IH(3,9)$ a hyperplane section of $I(3,9)$, of dimension $14$.
\item[] $X_1=I(3,8)$ the symplectic Grassmannian $I_1 \Gr(3,V_8)$, of dimension $12$.
\item[] $Z_1 \subset Y_1$ the exceptional locus of the projection $\pi':Y_1 \dashrightarrow I(3,8)$. Then $Z_1 \simeq I_2(2,8)$, of codimension $4$ in $Y_1$.
\item[] $Y_2=IH_2(3,8)$ a hyperplane section of $I_2(3,8)$, of dimension $8$.
\item[] $T=T(2,10)$, of dimension $8$.
\item[] $P=P(1,10) \subset \PP^9$, of dimension $6$, the so-called Peskine variety.
\end{itemize}

Note that all these varieties are smooth in general. 
Let us draw the following diagram, with all the correspondences we can connect to $Y$:

\begin{equation}\label{eq:the-nested-diagram}
\xymatrix{
   & E  \ar@{}[dr]|{(2)} \ar@{^{(}->}[r]^{cdim 7} \ar[d]^{\PP^7} & q^*Y  \ar[d]^{\PP^6} \ar[dr]^{\PP^2} & &\mathrm{Bl}_{Z}Y \ar@{}[dr]|{(3)} \ar[dl]_{bu} \ar[d]^{\PP^2} & F_1 \ar@{_{(}->}[l]_{cdim 3} \ar[d]^{\PP^3} & \\
   & T \ar@{^{(}->}[r]& \Gr(2,10) & Y & \Gr(3,9) & Y_1 \ar@{_{(}->}[l] &  \\
 E' \ar@{}[dr]|{(1)} \ar[d]_{\PP^2} \ar@{^{(}->}[r]^{exc.div.} & q^*T \ar[d]^{bu} \ar[u]_{\PP^1} & & & & \mathrm{Bl}_{Z_1} Y_1 \ar[u]^{bu} \ar@{}[dr]|{(4)} \ar[d]_{\PP^2} & F_2 \ar@{_{(}->}[l]_{cdim 3} \ar[d]^{\PP^3} \\
P \ar@{^{(}->}[r]_{cdim 3} & \PP^9& & & & X_1 & Y_2,\ar@{_{(}->}[l]  
}
\end{equation}

\noindent where the maps marked with $bu$ are blow-ups, the markings $\PP^n$ denote the (general) fiber over the corresponding locus, the marking $exc.div.$ stands for the embedding of the exceptional divisors, and the markings $cdim x$ stands for an embedding of a codimension $x$ locus.

Recall that for the last projection $Y_1 \dashrightarrow X_1$ to give rise to diagram $(4)$, we need to choose the center $V_1$ of the projection to be the kernel of the $2$-form $\omega_1$ defining the symplectic Grassmannian $I(3,9)$ whose hyperplane section is $Y_1$. 

\subsection{Hodge theoretical results}
We can use the correspondences in \eqref{eq:the-nested-diagram} to show that the K3 Hodge structure of
$Y$ spreads in the other Fano varieties of K3 type.

\begin{theorem}\label{K}
The Hodge structure $K$ is the minimal weight $2$ Hodge structure containing $H^{*-1,*+1}$ in the following Hodge structures:
\begin{itemize}
\item $H^{14}(Y_1,\CC)$,
\item $H^8(Y_2,\CC)$,
\item $H^j(T,\CC)$, for $j=6,8,10$,
\item $H^j(P,\CC)$, for $j=4,6,8$.
\end{itemize}
Moreover, $H^{p,q}(\bullet)/ K = 0$ for $p\neq q$ for $\bullet$ either $Y_1$, $Y_2$, $T$ or $P$.  In particular, $Y_1$ and $Y_2$ are Fano of pure K3 type, while $P$ and $T$ are
of non pure K3 type.

Finally, if $Y$ is very general, then $K$ coincides with the vanishing cohomologies of all of the above cohomology groups for $Y_1$, $Y_2$, and for $T$ if $j=6,10$.
\end{theorem}

\begin{proof}
The proof is obtained by using Propositions \ref{prop:Hodge-projection}, \ref{prop:Hodge-jump} and
\ref{prop:peski-Hodge} along the diagram \eqref{eq:the-nested-diagram}, and by the analysis of the Hodge numbers of the varieties involved.

Let us start with diagram (3). Proposition \ref{prop:Hodge-projection} gives an isomorphism of integral Hodge structures:
$$
H^{j-6}(Y_1,\CC)(-3) \oplus \bigoplus_{i=0}^2 H^{j-2i}(\Gr(3,9),\CC)(-1) \simeq
H^j(Y,\CC) \oplus \bigoplus_{i=0}^{6}H^{j-2i}(Z,\CC)(-i).
$$
On the left hand side, we notice that $H^{p,q}(\Gr(3,9))=0$ whenever $p \neq q$.
Similarly, on the right hand side $H^{p,q}(Z)=0$ whenever $p \neq q$, since $Z$ is isomorphic to a hyperplane
section of $\Gr(2,9)$ which is nothing but the symplectic Grassmannian $I(2,9)$.
It follows that $H^{9,11}(Y) \simeq H^{6,8}(Y_1)$, and hence that $K$ is the smallest sub-Hodge structure
of $H^{20}(\mathrm{Bl}_Z Y)$ containing them. The rest of the proof follows by comparison of Hodge numbers.

A similar argument applies to $Y_2$ using diagram (4): it is enough to notice that both $H^{p,q}(X_1)$ and $H^{p,q}(Z_1)$ are trivial whenever $p \neq q$, since $X_1$ is again a symplectic Grassmannian,
and $Z_1$ is isomorphic to a double hyperplane section of $\Gr(2,8)$ \cite{vlad-bi}.

Now consider diagram (2). Thanks to Proposition \ref{prop:Hodge-jump}, we have an isomorphism of integral Hodge structures:
\begin{equation}\label{eq:hodge-Y-and-T}
\bigoplus_{i=0}^2 H^{j-2i}(Y,\CC)(-i) \simeq \bigoplus_{i=0}^6 H^{j-2i}(\Gr(2,10),\CC)(-i) \oplus H^{j-14}(T,\CC),
\end{equation}
from which we can compute the Hodge numbers of $T$ (see also \cite[Proposition 3.27]{fm}).
Since $H^{p,q}(\Gr(2,10))=0$ whenever $p\neq q$, we deduce that $H^{9,11}(Y)(-i) \simeq H^{2+i,4+i}(T)$ for $i=0,1,2$. Hence $K$ is the smallest sub-Hodge structure of $H^{20}(q^* Y,\CC)$ containing $H^{2,4}(T)$, and similarly for $H^{3,5}(T) \subset H^{22}(q^*Y,\CC)$ and $H^{4,6}(T) \subset H^{24}(q^*Y,\CC)$.
The rest of the proof follows by comparison of Hodge numbers.

Finally, consider diagram (1). Proposition \ref{prop:peski-Hodge} gives an isomorphism of integral Hodge structures:
\begin{equation}\label{eq:hodge-P-and-T}
H^j(T,\CC) \oplus H^{j-2}(T,\CC)(-1) \simeq H^j(\PP^9) \oplus H^{j-2}(P,\CC)(-1) \oplus H^{j-4}(P,\CC)(-2).
\end{equation}

Knowing the Hodge numbers of $T$, we deduce that for $p \neq q$,  $H^{p,q}(q^*T)\neq 0$ is possible only when $p+q$ is $6$, $8$, $10$ or $12$. Moreover, since $H^{p,q}(\PP^n)=0$ for $p \neq q$, we get the following numerology:
$$
\begin{cases}
1 = h^{2,4}(T)\phantom{+h^{2,4}(T)} = h^{1,3}(P)+h^{0,2}(P)\\
2 = h^{3,5}(T)+h^{2,4}(T) = h^{2,4}(P)+h^{1,3}(P)\\
2 = h^{4,6}(T)+h^{3,5}(T) = h^{3,5}(P)+h^{2,4}(P)\\
1 = h^{4,6}(T)\phantom{+h^{2,4}(T)} = h^{4,6}(P)+h^{3,5}(P).
\end{cases}
$$
Then we obtain $h^{0,2}(P)=h^{4,6}(P)=0$, and $H^{1+i,3+i}(P) \simeq H^{2+i,4+i}(T)$, and the rest of the proof follows. 

\medskip

Recall that if $Y$ is very general, then $K$ coincides with the vanishing cohomology of $H^{20}(Y,\CC)$,
and is hence $22$-dimensional. By comparison of dimensions (see Table \ref{tab}) the vanishing cohomology
of $Y_1$, $Y_2$ and $T$ (in the appropriate degrees) is also at most $22$-dimensional.
We conclude by the
simplicity of $K$.
\end{proof}

\begin{table}[]
\centering

\begin{tabular}{||l | c c c c ||l | c c c c ||l | c c c c ||l | c c c c || l | c c c c ||}
\hline \hline
$h^0$ &  & 1 & & & & & & & & & & & & & & & & & & & & & &\\
$h^2$ &  & 1 & & & & &  & & & & & & & & & & & &  & & & & & \\
$h^4$  &  & 2 & & & & &  & & & & & & & & & & & & & & & & & \\
$h^6$  &  & 3 & & & $h^0$  &  & 1 & & & &  & & & & & & & & & & & & & \\
 $h^8$ &  & 4 & & & $h^2$ &  & 1 & & & & & &  & & & & &  & & & & & &\\
 $h^{10}$ &  & 5 & & &  $h^4$ & & 2 & & & & & &  & & & & & & & & & & &\\
$h^{12}$  &  & 7 & & &  $h^6$  & & 3 & & & $h^0$ & & 1 & & & $h^0$ & & 1 & & &  & & & &\\
$h^{14}$  & & 8 & & & $h^8$  & & 4 & & & $h^2$&  & 1 & & & $h^2$&  & 1 & & & $h^0$ & & 1 & &\\
$h^{16}$   & & 9 & & &$h^{10}$ &  & 5 & & & $h^4$ & & 2 & & & $h^4$ & & 2 & & & $h^2$ & & 1 & &\\
$h^{18}$   & & 10 & & &$h^{12}$  &  & 6 & & &$h^6$  & & 6 & & & $h^6$ & 1 & 22 & 1&  &$h^4$ & 1 & 22 & 1 &\\
 $h^{20}$  & 1 & 30 & 1 & & $h^{14}$  & 1 & 26 & 1 & & $h^8$  & 1 & 26 & 1  & & $h^8$ & 1 & 23 & 1 &
 &$h^6$ & 1 & 22 & 1 &\\
  \hline
\multicolumn{5}{||c||}{\multirow{2}{*}{$Y$}} & \multicolumn{5}{c||}{\multirow{2}{*}{$Y_1$}} & \multicolumn{5}{c||}{\multirow{2}{*}{$Y_2$}}& \multicolumn{5}{c||}{\multirow{2}{*}{$T$}} & \multicolumn{5}{c||}{\multirow{2}{*}{$P$}}\\
\multicolumn{5}{||c||}{} & \multicolumn{5}{c||}{} & \multicolumn{5}{c||}{} & \multicolumn{5}{c||}{} & \multicolumn{5}{c||}{} \\ \hline \hline
\end{tabular}
\caption{The nontrivial Hodge numbers of the varieties in diagram \eqref{eq:the-nested-diagram}.}
\label{tab}
\end{table}

It would be natural to conjecture that, in the very general case, $K$ also gives the primitive
cohomology of $H^j(P,\CC)$ for $j=4,6,8$. However such groups are $24$-dimensional (see Table \ref{tab}),
and $P$ sits in $\PP^9$, so that there is only one natural cycle coming from the ambient variety,
namely the hyperplane section.

This leads us to wonder whether there exists an algebraic cycle $A \subset P$ of dimension 
$4$, not homologous to a linear section. Such a cycle would indeed give a primitive class
$[Z]$ in $H^8(P,\ZZ)$ and therefore in $H^6(P,\ZZ)$ and also, by duality, in $H^4(P,\ZZ)$. 
One way to obtain such a  cycle could be the following: a point in $P \subset \PP^9$ is a line
$l \subset V_{10}$ such that the form $\Omega(l,\bullet,\bullet)$ has a four dimensional 
kernel $U_l$ (that contains $l$). This defines a natural map $\phi: P \to \Gr(4,V_{10})$,
and we could pull-back some Schubert cycles.

\smallskip

\begin{remark}
It would be interesting to relate the period maps for the varieties $Y$, $Y_1$ and $Y_2$.
Recall that at the infinitesimal level the local Torelli theorem asks for the natural map
$$ H^1(Y_i, T_{Y_i}) \longrightarrow Hom(H^{p+1, p-1}(Y_i), H^{p,p}(Y_i))$$
to be injective, where $Y_i$ is any of the three varieties above and dim $Y_i=2p$. Recall that in each of these three cases $H^{p+1, p-1}(Y_i) \cong \C$. For $Y$ the deformation space has dimension $20$, and $h^{10,10}(Y)=30$. The period map can therefore be injective. Moreover $H^1(T_Y) \cong H^{10,10}_{van}(Y)$, as follows for example from the Jacobian--type ring description of the cohomology ring of $Y$, see \cite{eg2}. For $Y_1$ and $Y_2$ the situation is slightly different. In both cases we have $h^{p,p}(Y_i)=26$ (and the vanishing subspace is $20$--dimensional), but we can compute that $h^1(T_{Y_1})=29$ and $h^1(T_{Y_2})=28$. Therefore there is no hope for the period map to be a local isomorphism. 

However, in both cases our construction gives a partial description of the deformation space of $Y_i$ in terms of $H^1(T_Y)$. 
In fact the deformation spaces of $Y=Y_0, Y_1, Y_2$ can be computed through their normal
exact sequences, which yield exact sequences of cohomology groups, for $i=0,1,2$:
$$0\lra H^0(Y_i, T\Gr(3,V_{10-i})_{|Y_i})\lra H^0(Y_i, \ko_{Y_i}(1)\oplus (\wedge^2\ku^*_{|Y_i})^{\oplus i})\lra H^1(Y_i, TY_i).$$
Moreover $H^0(Y_i, T\Gr(3,V_{10-i})_{|Y_i})\simeq H^0(\Gr(3,V_{10-i}), T\Gr(3,V_{10-i}))=
sl(V_{10-i})$. On the other hand, if $Y_i$ is defined by the $3$-form $\Omega_i$ and the $2$-forms
$\omega_1, \ldots , \omega_i$, one has 
$$H^0(Y_i,\wedge^2\ku^*_{|Y_i})=\wedge^2V_{10-i}^*/\langle \omega_1, \ldots , \omega_i\rangle, $$
$$H^0(Y_i,\ko_{Y_i}(1))=\wedge^3V_{10-i}^*/\langle \Omega_i, V_{10-i}^*\wedge \omega_1, \ldots , V_{10-i}^*\wedge\omega_i\rangle . $$
This implies the following descriptions of the infinitesimal deformation spaces:
$$H^1(T_{Y_0}) \simeq \wedge^3V_{10}^*/\langle X\Omega_0, X\in gl(V_{10})\rangle ,$$
$$H^1(T_{Y_1}) \simeq (\wedge^3V_{9}^*\oplus \wedge^2V_{9}^*)/\langle \Omega_1, \omega_1,
V_9^*\wedge \omega_1, X\Omega_1+  X\omega_1, X\in gl(V_{9})\rangle ,$$
$$
H^1(T_{Y_2}) \simeq (\wedge^3V_{8}^*\oplus \wedge^2V_{8}^*\oplus \wedge^2V_{8}^*)/
\langle \Omega_2, (\omega_1,0), (0,\omega_2),
V_8^*\wedge \omega_1, V_8^*\wedge \omega_2, X\Omega_2+  X(\omega_1,\omega_2), X\in gl(V_{8})\rangle .$$
Decomposing $V_{10}$ as $V_1\oplus V_9$ and $\Omega_0$ as $\Omega_1+e_1^*\wedge\omega_1$, 
the block decomposition of $gl(V_{10})$ yields $\langle X\Omega_0, X\in gl(V_{10})\rangle=
\langle \Omega_1, \omega_1, V_9^*\wedge \omega_1, V_9\dashv\Omega_1,  X\Omega_1+  X\omega_1, X\in gl(V_{9})\rangle$. We deduce a natural exact sequence 
$$0\lra  V_9\dashv\Omega_1\lra H^1(T_{Y_1})\lra H^1(T_{Y_0})\lra 0,$$
where $V_9\dashv\Omega_1 \subset  \wedge^2V_{9}^*$ is the space of $2$-forms obtained by 
contracting the $3$-form $\Omega_1$ with some vector in $V_9$. Similarly, we get 
the natural  exact sequence 
$$0\lra  V_8\dashv\Omega_2\lra H^1(T_{Y_2})\lra H^1(T_{Y_0})\lra 0.$$
\end{remark}

\subsection{A categorical counterpart}
Now we turn to derived categories. In this frame, moving the subcategory $\cat{A}$ around the diagram is much more complicated, due to the huge number of exceptional objects involved in semiorthogonal decompositions, and the titanic task of mutating such exceptional collections one to another. Hence we only have evidences but no proof for the following conjecture.

\begin{conjecture}\label{conj:decompositions}
Let $\cat{A}$ be the K3 subcategory of $\Db(Y)$ obtained as a semiorthogonal complement of 108 exceptional objects as in \eqref{eq:sod-for-DV}. Then we have (up to equivalences) the following semiorthogonal decompositions: 
\begin{equation}
\begin{array}{ll}
 \Db(Y_1)&=\sod{\cat{A}, 48 \text{ exceptional objects }}\\ 
 \Db(Y_2)&=\sod{\cat{A}, 24 \text{ exceptional objects }}\\
 \Db(T)&=\sod{\cat{A},\cat{A},\cat{A}, 9 \text{ exceptional objects }}\\
  \Db(P)&=\sod{\cat{A},\cat{A},\cat{A}, 4 \text{ exceptional objects }}\\
\end{array}
\end{equation}
In particular, $Y_1$ and $Y_2$ are of derived pure K3-type while $P$ and $T$ are of
derived non-pure K3 type.
\end{conjecture}

The main evidences of the conjecture are the following comparisons of semiorthogonal
decompositions based on correspondences from diagram \eqref{eq:the-nested-diagram}.

\begin{prop}\label{prop:decos-in-the-diagram}

\begin{itemize}
\item[A)]
We have the following decompositions:
$$\begin{array}{rl}
\Db(\mathrm{Bl}_Z Y) &= \sod{\Db(Y), \Db(Z)_1,\ldots,\Db(Z)_6}=\\
&= \sod{\Db(Y_1),\Db(\Gr(3,9))_1,\Db(Gr(3,9))_2,\Db(\Gr(3,9))_3},
\end{array}$$
where $\Db(Z)_i$ and $\Db(\Gr(3,9))_i$ are equivalent to $\Db(Z)$ and
$\Db(\Gr(3,9))$ for any $i$ respectively.

In particular, the first decomposition gives $300$ exceptional objects in $\Db(\mathrm{Bl}_Z Y)$
whose orthogonal complement is $\cat{A}$, while the second one gives $252$ exceptional objects
whose orthogonal complement is $\Db(Y_1)$.

\medskip

\item[B)] We have the following decompositions:
$$\begin{array}{rl}
\Db(\mathrm{Bl}_{Z_1} Y_1) &= \sod{\Db(Y_1), \Db(Z_1)_1,\ldots,\Db(Z_1)_3}=\\
&= \sod{\Db(Y_2),\Db(X_1)_1,\Db(X_1)_2,\Db(X_1)_3},
\end{array}$$
where $\Db(Z_1)_i$ and $\Db(X_1)_i$ are equivalent to $\Db(Z_1)$ and
$\Db(X_1)$ for any $i$ respectively.

In particular, the first decomposition gives $66$
exceptional objects in $\Db(\mathrm{Bl}_{Z_1} Y_1)$
whose orthogonal complement is $\Db(Y_1)$, while the we expect the second one 
to have $96$ exceptional objects in the orthogonal complement of $\Db(Y_2)$.

\medskip
\item[C)] We have the following decompositions:
$$\begin{array}{rl}
\Db(q^* Y) &= \sod{\Db(Y)_1, \Db(Y)_2,\Db(Y)_3}=\\
&= \sod{\Db(T),\Db(\Gr(2,10))_1,\ldots,\Db(\Gr(2,10))_7},
\end{array}$$
where $\Db(Y)_i$ and $\Db(\Gr(2,10))_i$ are equivalent to $\Db(Y)$ and
$\Db(\Gr(2,10))$ for any $i$ respectively.

In particular, the first decomposition gives $324$ exceptional objects in $\Db(q^* Y)$
whose orthogonal complement is generated by three copies of $\cat{A}$, while the second one gives
$315$ exceptional objects whose orthogonal complement is $\Db(T)$.
\medskip

\item[D)] We have the following decompositions:
$$\begin{array}{rl}
\Db(q^* T) &= \sod{\Db(T)_1, \Db(T)_2}=\\
&= \sod{\Db(P)_1,\Db(P)_2,\Db(\PP^9)},
\end{array}$$
where $\Db(T)_i$ and $\Db(P)_i$ are equivalent to $\Db(T)$ and
$\Db(P)$ for any $i$ respectively.

In particular, the second decomposition gives $10$ exceptional objects
whose orthogonal complement is generated by two copies of $\Db(P)$.
\end{itemize}
\end{prop}

\begin{proof}
\begin{itemize}
\item[A)] The decompositions are special cases of the blow-up formula and, respectively,
Corollary \ref{cor:ourcases} (see Lemma \ref{lem:norm-proj} for the calculation of the normal bundle of $F_1$) applied to (3) in diagram
\eqref{eq:the-nested-diagram}. The exceptional
objects counting comes from the fact that $\cat{A}$ is the complement of $108$ exceptional objects
in $\Db(Y)$, while $\Db(Z)$ is generated by $32$ exceptional objects by homological projective
duality \cite[Thm. 4.33]{rennemo-segal}, since $Z$ is isomorphic to a hyperplane
section of $\Gr(2,10)$ . On the other hand, $\Db(\Gr(3,9))$ is generated
by $84$ exceptional objects.

\item[B)] The decompositions are special cases of the blow-up formula and, respectively,
Corollary \ref{cor:ourcases} (see Lemma \ref{lem:norm-proj2} for the calculation of the normal bundle of $F_2$) applied to (4) in diagram
\eqref{eq:the-nested-diagram}. The exceptional
objects counting comes from the fact that $\Db(Z_1)$ is generated by $22$ exceptional objects, 
 by (incomplete) homological projective duality \cite[Thm. 4.33]{rennemo-segal},
since it is isomorphic to a double hyperplane
section of $\Gr(2,9)$ and odd Pfaffians have codimension
3 so that the projective dual of $Z_1$ is empty. On the other hand, $\Db(X_1)$ is
expected to be generated by $32$ exceptional objects.

\item[C)] The decompositions are special cases of the projective bundle formula and, respectively,
Corollary \ref{cor:ourcases} (see Lemma \ref{lem:norm-FL23} for the calculation of the normal
bundle of $F_2$) applied to (2) in diagram \eqref{eq:the-nested-diagram}. The exceptional
objects counting comes from the fact that $\cat{A}$ is the complement of $108$ exceptional objects
in $\Db(Y)$, and $\Db(\Gr(3,9))$ is generated
by $45$ exceptional objects.

\item[D)] The decompositions are special cases of the projective bundle formula, and, respectively,
blow-up formula applied to (1) in diagram
\eqref{eq:the-nested-diagram}.
\end{itemize}
\end{proof}

Proposition \ref{prop:decos-in-the-diagram} gives numerical evidences since it allows to count the
number of exceptional objects and copies of $\cat{A}$ one expects. The proof of Conjecture
\ref{conj:decompositions} could now follow by mutating the exceptional objects in the different decompositions.
This is a very hard task, due to the high number of objects. Moreover, to the best of 
the authors' knowledge,
there is no explicit description of exceptional collections of the required length on $Y_1$ and $Y_2$.
On the other hand, in the case of $T$ and $P$, we can provide explicit collections.

\begin{prop}\label{prop:exc-coll-on-T}
The collection
$$\{\ko,\ku^*,S^2\ku^*,\ko(1),\ku^*(1),S^2\ku^*(1),\ko(2),\ku^*(2),S^2\ku^*(2)\}$$
is exceptional on $T$.
\end{prop}

\begin{proof}
First, recall that $T$ is cut on $\Gr(2,10)$ by a general global section of the vector bundle $\kq^*(1)$. The associated Koszul complex is  \begin{equation}\label{eqn:koszul}
0 \to \mathrm{det}(\kq(-1)) \to \W^7\kq(-1) \to \ldots \to \kq(-1) \to \OO \to \OO_T \to 0.
\end{equation}
Therefore to calculate the cohomology groups of any bundle $\mathcal{F}_T$ restricted to $T$ it will suffice to tensor the above complex with $\mathcal{F}$. The cohomology groups of $\mathcal{F}$ on $\Gr(2,10)$ can be computed using the Bott--Borel--Weil (BBW) theorem. The decomposition into irreducible components of every bundle involved will be deduced from the Littlewood-Richardson formula. In fact they will all be twists of symmetric powers of $\ku$, so the special case of BBW that will be useful to us is the following:

\begin{lemma}\label{BBW}
Suppose $S^p\ku\otimes\wedge^q\kq (-i)$ is not acyclic on $G(2,10)$, where $q<8$. 
Then either
\begin{enumerate}
\item $i\ge 10$,
\item $p+i\le 0$, 
\item $p+q+i=9$ and $i\le 1$,
\item $q+i=10$ and $p+i\ge 10$.
\end{enumerate}
\end{lemma}

We will split the proof of the Proposition into three parts, 
checking first the exceptionality and then the additional required vanishings.
Let $\cat{E}:= \sod{\ko,\ku^*,S^2\ku^*} \subset \Db(T)$. 

\smallskip\noindent {\bf Step 1.}
First we prove that all the bundles in the collection are exceptional. To this end, it is enough to show that
the bundles $\ko$, $\ku^*$ and $S^2\ku^*$ are exceptional. Since $T$ is a Fano variety, then
$\ko$ is exceptional. The other two cases give:
\begin{itemize}
\item $\Hom^*(\ku^*,\ku^*) \simeq H^*(T,\ku \otimes \ku^*)$. 
\item $\Hom(S^2\ku^*,S^2\ku^*) \simeq H^*(T,S^2\ku \otimes S^2\ku^*)$. 
\end{itemize}

The bundles $\ku \otimes \ku^*$ and $S^2\ku \otimes S^2\ku^*$ are not irreducible:  
they split into $S^2 \ku(1)\oplus  \ko$ and  $S^4 \ku (2) \oplus S^2 \ku (1) 
\oplus \ko$, respectively. 
Using Lemma \ref{BBW} and the Koszul complex (\ref{eqn:koszul}), it is easy to check that the only non acyclic factor is $\ko$.

\smallskip\noindent {\bf Step 2.}
Now we verify the orthogonality of the bundles generating $\cat{E}$. This will imply
that every $\cat{E}(i)$ is generated by an exceptional collection of length 3.

There are three cases:

\begin{itemize}
\item $\Hom^*(\ku^*,\ko) \simeq H^*(T,\ku)$.
\item $\Hom^*(S^2\ku^*,\ko) \simeq H^*(T,S^2\ku)$.
\item $\Hom^*(S^2\ku^*,\ku^*) \simeq H^*(T,S^2\ku\otimes\ku^*)$.
\end{itemize}

The bundle $S^2\ku \otimes\ku^*$ splits into $S^3 \ku (1) \oplus \ku$. 
Using Lemma \ref{BBW} and the Koszul complex (\ref{eqn:koszul}), we check that $\ku$, $S^2\ku$ and $S^3 \ku (1)$  are all acyclic.

\smallskip\noindent {\bf Step 3.}
There remains to check the orthogonality of the bundles generating $\cat{E}$ with those generating $\cat{E}(i)$ for $i=1,2$.

The orthogonality $\Hom(\ko(i),\ko)=0$ follows from Kodaira vanishing since $T$ has index 3. Noticing that $\ku^*=\ku(1)$, the other cases give:

\begin{itemize}
\item $\Hom^*(\ko(i),\ku^*) \simeq H^*(T,\ku^*(-i))$,
\item $\Hom^*(\ko(i),S^2\ku^*) \simeq H^*(T,S^2\ku^*(-i))$,
\item $\Hom^*(\ku^*(i),\ko) \simeq H^*(T,\ku(-i))$,
\item $\Hom^*(\ku^*(i),\ku^*) \simeq H^*(T,\ku(-i) \otimes \ku^*)$, and $\ku(-i) \otimes \ku^* \simeq S^2 \ku (1-i) \oplus \ko(-i)$,
\item $\Hom^*(\ku^*(i),S^2\ku^*) \simeq H^*(T,\ku(-i) \otimes S^2\ku^*)$, and $\ku(-i) \otimes S^2\ku^*$ splits into $S^3 \ku(2-i)\oplus \ku (1-i)$,
\item $\Hom^*(S^2\ku^*(i),\ko) \simeq H^*(T,S^2\ku(-i))$,
\item $\Hom^*(S^2\ku^*(i),\ku^*) \simeq H^*(T,S^2\ku(-i) \otimes \ku^*)$, and $S^2\ku(-i) \otimes \ku^*$ splits into $S^3\ku (1-i) \oplus \ku(-i)$,
\item $\Hom(S^2\ku^*(i),S^2\ku^*) \simeq H^*(T,S^2\ku \otimes S^2 \ku^* (-i))$, and $S^2\ku(-i) \otimes S^2\ku^*$ splits into $S^4\ku (2-i)\oplus \ku(1-i)\oplus \ko(-i)$.
\end{itemize}

So we are reduced to checking the acyclicity of $\ku(-j)$ for $j=0,1,2$, of $S^2\ku(-j)$ 
for $j=-1,0,1,2$,  of $S^3\ku(-j)$ for $j=-1,0,1$, and of $S^4\ku(-j)$ for $j=-1,0$. 
Again this is a straightforward application of Lemma \ref{BBW}.
\end{proof}

\smallskip
The Peskine variety $P\subset\PP^9$ is the locus where the section of $\wedge^2\kq^* (1)$
defined by the three-form $\Omega$ has rank at most six. For $\Omega$ general, this occurs
in codimension three, and the rank drops to four in codimension ten, hence nowhere, and 
$P$ is smooth of dimension six. Being a Pfaffian degeneracy locus, its structure sheaf admits 
the following resolution:
\begin{equation}\label{resP}
0\lra \cO(-7)\lra \kq(-4)\lra \kq^* (-3)\lra \cO\lra\cO_P\lra 0.
\end{equation}
In particular $\omega_P=\cO_P(-3)$.

\begin{prop}\label{prop:exc-coll-on-P}
The collection $\{\cO, \kq, \cO(1), \cO(2)\}$ is exceptional on $P$. 
\end{prop}

\begin{proof}
Since $\omega_P=\cO_P(-3)$, the sequence $\cO, \cO(1), \cO(2)$ is exceptional on $P$.
Let us prove that $\kq$ is exceptional; in otherwords, that $End_0(\kq)$ is acyclic on $P$. 
In order to
check this, we tensor out the sequence (\ref{resP}) by $End_0(\kq)$ and we use
the Bott-Borel-Weil theorem. On $\PP^9$, the latter  implies that for any sequence $\alpha 
=(\alpha_1\ge\cdots\ge\alpha_9)$, the bundle $S_\alpha \kq (-\ell)$ is acyclic if and only if 
there exists an integer $q$ such that $\alpha_{q}-q+10=\ell$. 
\begin{enumerate}
\item $End_0(\kq)$ corresponds to $\alpha = (1,0,...,0,-1)$ and is acyclic because 
$\alpha_{9}-9+10=0$. Similarly $End_0(\kq)(-7)$ is acyclic because $\alpha_3-3+10=7$. 
\item $End_0(\kq)\otimes \kq^* (-3)$ decomposes into three factors $S_\beta \kq(-3)$, $S_{\beta'}\kq(-3)$ and $S_{\beta''}\kq(-3)$, 
with  $\beta = (1,0,...,0,-1,-1)$, $\beta' = (1,0,...,0,-2)$ and $\beta'' = (0,0,...,0,-1)$; 
they are all acyclic because $\beta_7-7+10=\beta'_7-7+10=\beta''_7-7+10=3$. 
\item $End_0(\kq)\otimes \kq* (-4)$ gives three factors $S_\gamma \kq(-4)$, $S_{\gamma'}\kq(-4)$
and $S_{\gamma''}\kq(-4)$, 
with  $\gamma = (1,1,0,...,0,-1)$, $\gamma' = (2,0,...,0,-1)$ and $\gamma'' = (1,0,0,...,0)$; 
they are all acyclic because $\gamma_6-6+10=\gamma'_6-6+10=\gamma''_6-6+10=4$.
\end{enumerate}
This implies our claim that $End_0(\kq)$ is acyclic on $P$. There remains to check that 
$\kq^*$, $\kq(-1)$ and $\kq(-2)$ are acyclic on $P$, which is again a straightforward 
consequence of the Bott-Borel-Weil Theorem. 
\end{proof}

The nature of the above exceptional collections for $T$ and $P$ let us expect Conjecture \ref{conj:decompositions}
to be improved as follows.

\begin{conjecture}\label{conj:deco-improved}
\begin{enumerate}
\item [T)] There is a fully faithful functor $\Phi: \cat{A} \to \Db(T)$, so that
$$\cat{B}=\sod{\Phi \cat{A},\ko,\ku^*,S^2\ku^*} \subset \Db(T)$$
provides a rectangular Lefschetz decomposition:
$$\Db(T)=\sod{\cat{B},\cat{B}(1),\cat{B}(2)}.$$

\item [P)] There is a fully faithful functor $\Psi: \cat{A} \to \Db(P)$, so that
$$\cat{C}_1=\sod{\Psi \cat{A}, \ko} \subset \cat{C}_0 = \sod{\Psi\cat{A},\cO, \kq} \subset \Db(P)$$
provides a Lefschetz decomposition:
$$\Db(P) = \sod{\cat{C}_0,\cat{C}_1(1),\cat{C}_1(2)}.$$
\end{enumerate}
\end{conjecture}

\begin{remark}
Notice that the projections and jumps considered here from diagram \eqref{eq:the-nested-diagram}
are not all the possible correspondences one can get starting from $Y$. First of all, one
could perform a $(4,3)$ jump to obtain that the variety $T(4,V_{10})$ has 7 copies
of the Hodge structure $K$ in different degrees, and, conjecturally, as many
copies of $\cat{A}$ in its derived category.

One can also project further down to $V_7$, but this would require to consider singular cases. 
Anyway, this projection is of major interest since it involves a K3 surface of degree 12 (a construction which was used in \cite{debarre-voisin} to show that a hyperk\"ahler arising as moduli space on $Y$ is deformation equivalent to a Hilbert scheme on such K3 surface), and
occurs in the following cases.
Write $V_{10}=V_7 \oplus W_3$ and take a 3-form on $V_{10}$ with projections zero on the components $V_7 \otimes \bigwedge^2 W_3$ and $\bigwedge^3 W_3$. This is a divisorial condition: it corresponds to a form in $\wedge^2V_7\wedge V_{10}$, a codimension $22$ subspace of $\wedge^3V_{10}$. Since we have a $21$ dimensional family of such spaces, they span a hypersurface $H$ in $\wedge^3V_{10}$. To be more precise, the orthogonal to $\wedge^2V_7\wedge V_{10}$ in $\wedge^3 V_{10}^*$ is $\wedge^2V_7^\perp\wedge V_{10}^*$, which is nothing but the affine tangent space to $\Gr(3,V_{10}^*)$ at $V_7^\perp$, so $H$ is the cone over the projective dual to 
$\Gr(3,V_{10}^*)$.
\end{remark}

\subsection{Normal bundles of special loci}
In this section we calculate the normal bundles of the special loci in diagram
\eqref{eq:the-nested-diagram}, so as to ensure that Corollary \ref{cor:ourcases} applies.
We keep the notations from diagram \eqref{eq:the-nested-diagram}.

\begin{lemma}\label{lem:norm-proj}
Consider the projective bundle $q: F_1=\PP(\cO\oplus\ku^*) \to Y_1$, and denote by $\kr$ the relative tautological quotient bundle. Then $\kn_{F_1/\mathrm{Bl}_{Z}Y} \simeq \kr^* \otimes q^* \ko(1)$.
\end{lemma}

\begin{proof}
Le us denote $\widetilde{Y}:=\mathrm{Bl}_{Z}Y$ and $\widetilde{G}:=\mathrm{Bl}_{\Gr(2,9)}\Gr(3,10)$. Consider the diagram
$$
\xymatrix{
 & Y \ar@{^{(}->}[r] & \Gr(3,10) \\
F_1 \ar[d]_{q} \ar@{^{(}->}[r] & \widetilde{Y}\ar[d]^p \ar@{^{(}->}[r] \ar[u]^\sigma &
\widetilde{G} \ar[d]^\pi \ar[u]_\tau \\
Y_1 \ar@{^{(}->}[r] & \Gr(3,9) \ar[r]^{=} & \Gr(3,9),
}
$$
where $\sigma$ and $\tau$ are the blow-ups, and both $\pi$ and $q$ are the
$\PP^3$-bundles obtained from the projectivization
of the rank 4 bundle $\ke:=\ko \oplus \ku^*$.
The middle line gives a nested sequence for the normal bundles:
$$0 \longrightarrow \kn_{F_1/\widetilde{Y}} \longrightarrow \kn_{F_1/\widetilde{G}}
\longrightarrow (\kn_{\widetilde{Y}/\widetilde{G}})_{\vert F_1} \longrightarrow 0.$$
Note that $Y_1 \subset \Gr(3,9)$ is the zero locus of a regular section of
 $\bigwedge^2 \ku^* \oplus \ko(1)$. Equivalently the first bundle can be seen as $\ku(1)$.
Since $q$ is nothing but the restriction of $\pi$, we deduce that
$$\kn_{F_1/\widetilde{G}} = q^* \kn_{Y_1/\Gr(3,9)} = q^* (\ku(1) \oplus \ko(1)).$$
On the other hand, $Y \subset \Gr(3,10)$ is a hyperplane section, so its normal bundle
is $\ko(1)$. Hence $\kn_{\widetilde{Y}/\widetilde{G}}=\sigma^* \ko(1)$.
Now we notice that $\sigma^*\ko(1)=\pi^*\ko(1) \otimes \ko_{\pi}(1)$ so that
$$(\kn_{\widetilde{Y}/\widetilde{G}})_{\vert F_1}=\sigma^* \ko(1)_{\vert F_1} =
q^*\ko(1) \otimes \ko_q(1).$$
The nested sequence for normal bundles turns then out to be nothing but the dual of the
relative tautological sequence for the projective bundle $q : F_1=\PP(\ko\oplus\ku^*)\to Y_1$, 
up to a shift by $q^*\ko(1)$.
\end{proof}

The same techniques allow us to calculate the normal bundle of the special locus of
the second projection.

\begin{lemma}\label{lem:norm-proj2}
Consider the projective bundle $q: F_2=\PP(\ko\oplus\ku^*) \to Y_2$, and denote by $\kr$ the relative tautological quotient bundle of this fibration.  
Then $\kn_{F_2/\mathrm{Bl}_{Z_1}Y_1} \simeq \kr^* \otimes q^* \ko(1)$.
\end{lemma}

Finally, let us compute the normal bundle of the exceptional locus $E$ of 
diagram \eqref{eq:the-nested-diagram}. 
\begin{lemma}\label{lem:norm-FL23}
Consider the projective bundle $\pi : E=\PP(V_{10}/\ku_2)\rightarrow T\subset G(2,10)$, 
and denote by $\kr$ the relative tautological quotient bundle.
Then $N_{E/q^*Y} \simeq \kr^*  \otimes \pi^* \ko(1)$.
\end{lemma}

\begin{proof}
Denote by $p$ the projection from $q^* Y \to \Gr(2,10)$, so that we have a diagram:
$$\xymatrix{
& Y \ar@{^{(}->}[r] & \Gr(3,10) \\
E \ar@{^{(}->}[r] \ar[d]^\pi & q^* Y\ar@{^{(}->}[r] \ar[u]^q \ar[d]^p & \mathrm{Fl}(2,3,10) \ar[u]^q \ar[d]^\rho\\
T \ar@{^{(}->}[r] & \Gr(2,10) \ar[r]^{=} & \Gr(2,10).
}$$
where both $\pi$ and $\rho$ are the projective bundles obtained
by the projectivization of the rank 8 vector bundle $\kq=V_{10}/\ku_2$.
The middle line gives a nested sequence of normal bundles:
$$0 \longrightarrow \kn_{E/\widetilde{q^*Y}} \longrightarrow \kn_{E/\mathrm{Fl}(2,3,10)}
\longrightarrow (\kn_{q^*Y/\mathrm{Fl}(2,3,10)})_{\vert E} \longrightarrow 0.$$
Note that $T \subset \Gr(2,10)$ is the zero locus of a regular section of
$\kq^*(1)$. Since $\pi$ is nothing but the restriction of $\rho$,
we deduce that
$$\kn_{E/\mathrm{Fl}(2,3,10)} = \pi^* \kn_{T/\Gr(2,10)} = \pi^* \kq^*(1)= \pi^*\kq^* \otimes \pi^*\ko(1).$$
On the other hand, $Y \subset \Gr(3,10)$ is a hyperplane section, so its normal bundle
is $\ko(1)$. Hence $\kn_{q^*Y/\mathrm{Fl}(2,3,10)}=q^* \ko(1)$. Notice that $q^* \ko(1) =
\ko_{\rho}(1) \otimes \rho^*\ko(1)$, so that:
$$(\kn_{q^*Y/\mathrm{Fl}(2,3,10)})_{\vert E} = q^* \ko(1)_{\vert E} = \pi^* \ko(1) \otimes \ko_\pi(1).$$
The nested sequence for normal bundles turns then out to be dual to the relative
tautological sequence for the projective bundle $E=\PP (\kq) \to T$, up to a shift
by $\pi^*\ko(1)$.
\end{proof}

\section{On Coble cubics}
A nested construction, similar to the one treated in details in Section \ref{sect:DV} can
be carried over for a linear section $Y$ of $\Gr(3,V_n)$, for any $n$. If $n \geq 10$, such
a $Y$ would be Fano of $(n-8)$-Calabi-Yau type, and the Calabi-Yau structure spreads around
the different varieties in the diagram, as soon as one can guarantee the smoothness.
Going through the general case would be too complicated and out of the scope of this paper.
We present in this section the case $n=9$,  and make a short remark on the case $n=11$.

\subsection{Linear section of $\Gr(3,9)$, a weight one Hodge structure and the Coble cubic}\label{sect:coble}
The hyperplane section $T(3,9) \subset \Gr(3,V_9)$ carries a weight one Hodge structure in its middle cohomology
$H^{17}(T(3,9),\CC) = H^{9,8}(T(3,9)) \oplus H^{8,9}(T(3,9))$, which is $4$-dimensional. This weight one Hodge structure
is then similar to the one of a genus $2$ curve, and we can carry either projections to $\Gr(3,V_n)$ with
$n < 9$ or jumps to $\Gr(k,V_9)$ with $k < 3$.

In the first case, we can see that the weight one Hodge structure is carried to
$HI(3,8)$ which is an 11-dimensional Fano variety. If we want to push this further to
$HI_2(3,7)$ (which is a 5-dimensional Fano variety), we need to project along a line in the kernel of the 2-form defining $HI(3,8)$, which would then be singular in this case.

The case of jumps is probably more interesting, since if we perform twice this correspondence, we finally get
to Coble cubic hypersurfaces in $\PP^8$. We focus on these two correspondences. Let us first fix the
following notations.

\begin{itemize}
\item[] $X= T(3,9)$ the hyperplane section of $\Gr(3,V_9)$, smooth of dimension $17$.
\item[] $W=T(2,9)$, smooth of dimension $7$.
\item[] $C=P(1,9) \subset \PP^8$, of dimension $7$, the Coble cubic.
\item[] $S \subset C$ is the singular locus of $C$, an abelian surface.
\end{itemize}
That $P(1,9) \subset \PP^8$ is a Coble cubic was first observed in \cite{gsw}, section $5$. 
Its traditional
characterization is that given a $(3,3)$-polarized abelian surface $S$, embedded in $\PP^8$
by the associated linear system, this is the unique cubic hypersurface that is singular 
exactly along $S$. For this result and a general introduction to the Coble hypersurfaces,
we refer to \cite{beauville-coble}. 
 
\smallskip
The $(1,2)$ and $(2,3)$ jumps give rise to the following diagram:
\begin{equation}\label{eq:diag-for-3,9}
\xymatrix{
& \bullet \ar[dl]_{\PP^3} \ar@{}[d]|{(1)} \ar@{^{(}->}[r]^{cdim 3} & q^*W \ar[dl]^{\PP^1} \ar[dr]^{\PP^1} & \bullet \ar@{}[dr]|{(2)} \ar@{^{(}->}[r]^{cdim 6} \ar[d]^{\PP^6}  & q^*X \ar[d]^{\PP^5} \ar[dr]^{\PP^2} & \\
S \ar@{^{(}->}[r] & C \ar@{^{(}->}[r] & \PP^8 & W \ar@{^{(}->}[r] & \Gr(2,9) & X  \ar@{^{(}->}[r] & \Gr(3,9),
}
\end{equation}
where we use the conventions we introduced for \eqref{eq:the-nested-diagram}. Using Proposition
\ref{prop:Hodge-jump} in diagram (2), and the fact that $H^{a,b}(\Gr(2,9))=0$ for $a \neq b$, we
get
$$\begin{array}{c}
h^{4,5}(W) = h^{3,4}(W) = h^{2,3}(W) = 2,\\
h^{a < b}(W) = 0 \text{ otherwise.}
\end{array}$$
On the categorical side, notice that a rectangular Lefschetz decomposition for $\Gr(3,9)$ is not known
so that we can only expect (for numerical reasons) the derived category of $X$ to be generated
by 
$74$ exceptional objects and the derived category of a genus two curve $\Gamma$. Indeed, the Euler characteristic
of $X$ is $72$, and the Euler characteristic of $\Gamma$ is $-2$.

Moreover, we expect the derived
category of $W$ to be generated by $6$ exceptional objects and three copies of $\Db(\Gamma)$. Indeed, one has that
the Euler characteristic of $W$ is $0$ as one can calculate from square (2) in \eqref{eq:diag-for-3,9}.

On the other hand, the two expectations are related by Proposition
\ref{prop:semiorth-for-Hsections-in-jumps} applied to square (2) in \eqref{eq:diag-for-3,9}. Indeed,
the $\PP^2$ bundle $q^*X \to X$ would provide $222$ objects in $\Db(q^*X)$. On the other hand $\Db(\Gr(2,9))$
is generated by $36$ objects which, via the (generic) $\PP^5$-bundle structure $q^* X \to \Gr(2,9)$
provide $216$ objects. It is not difficult to construct a length 6 exceptional collection on $W$.

\begin{prop}
The collection
$$\{ \ko,\ku^*,\ko(1),\ku^*(1),\ko(2),\ku^*(2)\}$$
is exceptional in $\Db(W)$.
\end{prop}

\begin{proof}
The proof is very similar to the one of Proposition \ref{prop:exc-coll-on-T}. First of all, it is easy to check
that both $\ko$ and $\ku^*$ are exceptional. To verify the required orthogonalities, we have to check
acyclicity of the following bundles on $W$:
\begin{enumerate}
\item $\ku(-i)$ for $i=0,1,2$,
\item $\ko(-i)$ for $i=1,2$,
\item $\ku^*(-i)$ for $i=1,2$, but note that $\ku^*(-2)=\ku(-1)$,
\item $(\ku^* \otimes \ku)(-i) = (S^2  \ku(1) \oplus \ko)(-i)$, for $i=1,2$.
\end{enumerate}
This can be performed via BBW or using the fact that $W$ is a Fano variety of index $3$.
\end{proof}
The shapes of the exceptional collection and of the Hodge structure of $W$ lead us to formulate
a conjecture which is very similar to Conjecture \ref{conj:deco-improved}, part T).

\begin{conjecture}\label{conj:deco-for-T29}
There is a fully faithful functor $\Phi: \Db(\Gamma) \to \Db(W)$, so that
$$\cat{B}=\sod{\Phi \Db(\Gamma),\ko,\ku^*} \subset \Db(W)$$
provides a rectangular Lefschetz decomposition:
$$\Db(W)=\sod{\cat{B},\cat{B}(1),\cat{B}(2)}.$$
\end{conjecture}

Considering diagram (1) in \eqref{eq:diag-for-3,9}, one cannot apply results describing decompositions of the
Hodge theory or the derived categories, since the cubic $C$ singular. All what we can say is via the
$\PP^1$-bundle $q: q^* W \to W$, that is, that both the derived category and the Hodge structure of $q^*W$
are given by two copies of those of $W$. On the other hand, we can still perform calculations in the
Grothendieck ring $K_0(\mathrm{Var}(\CC)))$ of complex varieties. Indeed, we have:
$$
[q^*W] = [W](1+\LL) = [C](1+\LL) + [S]\LL^2 (1+\LL).
$$
Supposing that $(1+\LL) = [\PP^1]$ is not a zero-divisor, we get:
\begin{equation}\label{eq:coble-inK0}
[W] = [C] + [S]\LL^2.
\end{equation}

First of all, recall that the Hodge structure and (conjecturally) the derived category
of $W$ are related to a genus 2 curve. The description of the class of $W$ on the right hand
side of \eqref{eq:coble-inK0} suggests a tight relationship between such a curve and the Abelian
variety $S$.

We can push this analysis further to propose a candidate for a crepant categorical resolution of singularities
of the Coble cubic $C$. Indeed, a generalization of Proposition \ref{prop:main-sod} would give a semiorthogonal
decomposition of $q^* W$ in two copies of $\Db(C)$ and two copies of $\Db(S)$, that is, $q^* W$ can be thought of
(homologically) as a $\PP^1$-bundle over a smooth category which would 'differ' from $\mathrm{Perf}(C)$ only
by a copy of its singular locus $S$. Then we could expect the following description for a categorical crepant
resolution of singularities of the Coble cubic.

\begin{conjecture}\label{conj:resol-Coble}
There are functors $\Psi_i: \Db(\Gamma) \to \Db(q^* W)$ for $i=1,2,3$ and exceptional objects $E_j$
for $j=1,\ldots,6$, so that the category
$$\widetilde{\cat{C}}= \sod{\Psi_1 \Db(\Gamma),E_1,E_2,\Psi_2 \Db(\Gamma),E_3,E_4,\Psi_3 \Db(\Gamma),E_5,E_6}$$
is a crepant categorical resolution of singularities of $C$.
\end{conjecture}
Note that the choice of distributing exceptional objects in the categorical resolution in Conjecture
\ref{conj:resol-Coble} is arbitrary, since one can act by mutations. But it suggests an even stronger
expectation, that is that one can have a crepant categorical resolution of singularities of $C$
carrying a length 3 rectangular Lefschetz decomposition.

\subsection{Resolving the Coble cubic}
In all the sequel we will consider varieties that are naturally embedded into partial 
flag varieties. We will denote by $\ku_d$ the rank $d$ tautological bundle on such a 
partial flag variety, as well as its restriction to a given subvariety (with the hope 
that this will not confuse the reader). 

A geometrical resolution of singularities of the Coble cubic can be obtained by the above construction
as follows. Let $\omega$ be a general 2-form on $V_9$, and $W_\omega$ the corresponding hyperplane section of
$W \subset \Gr(2,9)$. That is, $W$ is the locus of those $\omega$-isotropic planes $U_2$ such that
$\Omega(u,v,\bullet)=0$ for all vectors $u$, $v$ of $_2$. Restricting the $(1,2)$-jump
to $W_\omega$ gives rise to the following diagram:
\begin{equation}\label{eq:the-reso-of-coble}
\xymatrix{
& E \ar[dl]_\pi \ar@{^{(}->}[r]^j & q^*W_{\omega} \ar[dl]_p \ar[dr]^q &\\
C_\omega \ar@{^{(}->}[r] & C & & W_{\omega},
}
\end{equation}
where $q: q^*W_{\omega} \to W_{\omega}$ is a $\PP^1$-bundle, so that $q^*W_{\omega}$ is smooth. We are going
to describe the exceptional locus $E \to C_\omega$.
We claim that $p: q^* W_{\omega} \to C$ is a birational map. Indeed, $q^*W_{\omega}$ is the locus
of pairs $(U_1,U_2)$ with $U_2 \subset V_9$ a plane corresponding to a point in $W_{\omega}$ and 
$U_1 \subset U_2$ a line. The map $p$ projects the
pair $(U_1,U_2)$ to $U_1$, and since $\Omega(l,u,\bullet)=0$ for any $l\in U_1$ and $u\in U_2$, the two-form 
$\Omega(l,\bullet, \bullet)$ is degenerate. So the image of $q^*W_{\omega}$ by $p$ is contained in $C$.

Now, given a point in $C$, i.e. a line $U_1=\langle l\rangle\subset V_9$ such that the 2-form $\Omega_l:=\Omega(l,\bullet,\bullet)$
is degenerate, the fiber of $p$ over $U_1$ is the set of planes $U_2\supset U_1$ that belong to 
$W_{\omega}$, so this fiber is isomorphic to the projectivization of $(\ker \Omega_l \cap U_1^\perp)/U_1$ 
(where the orthogonality is taken with respect to the form $\omega$). 
There are three possibilities. 
\begin{itemize}
\item $\ker \Omega_l$ is three-dimensional and not contained in $U_1^\perp$. This is the general case, hence it defines a dense open subset $C_0$ of $C$. In this case $U_2$ must be equal to 
$\ker \Omega_l \cap U_1^\perp$, so $p$ is an isomorphism over $C_0$.
\item $\ker \Omega_l$ is three-dimensional and contained in $U_1^\perp$. This is a codimension two condition, we call the corresponding locus $C_1$. The fiber of $p$ over $U_1$ is then 
a projective line.
\item $\ker \Omega_l$ is five-dimensional, that is, $U_1$ belongs to $S$. This kernel cannot 
be contained in $U_1^\perp$  (this is a codimension four condition), so the fiber of $p$ is 
a projective plane.
\end{itemize}
In particular $p : q^*W_\omega\lra C$ is a resolution of singularities. We deduce:

\begin{prop}\label{coblerat}
The Coble cubic $C$ has rational singularities.
\end{prop}

\begin{proof}
Recall that $W_\omega$ is the zero-locus of a general section of the vector bundle $\ke=\kq^*(1)\oplus \ko(1)$ on $G(2,V_9)$. So $q^*W_\omega$ is the zero-locus of a general 
section of  $q^*\ke$ on the flag manifold $\mathrm{Fl}(1,2,V_9)$, and we can resolve its structure
sheaf by the Koszul complex
$$0\lra q^*\wedge^8\ke^*\lra\cdots\lra q^*\ke^*\lra \ko_{\mathrm{Fl}(1,2,V_9)}\lra\ko_{q^*W_\omega}\lra 0.$$
In order to prove that $R^ip_{*}\ko_{q^*W_\omega}=0$ for $i>0$, it is then enough to check
that for all $0\le j\le 8$ and $i>0$,  $R^{i+j}p_{*}q^*\wedge^{j}\ke^*=0$. Since the 
projection from $\mathrm{Fl}(1,2,V_9)$ to $\PP(V_9)$ is a fiber bundle (with fiber $\PP(V_9/L)$ over the 
point $[L]\in \PP(V_9)$), this vanishing can just be checked on each fiber, and we thus
need to verify that 
$$H^{i+j}(\PP(V_9/L), q^*\wedge^{j}\ke^*_{|\PP(V_9/L)})=0, \text{ for } i>0.$$
On $\PP(V_9/L)$ the tautological line bundle is $\ko(-1)=U_2/L$, and is isomorphic to the restriction of $q^*\ko(-1)$. Moreover the quotient bundle is also the restriction of $q^*\kq$. 
We deduce that $q^*\ke_{|\PP(V_9/L)})\simeq \ko(1)\oplus \kq(1)$, where now $\ko(1)$ and $\kq$
are the hyperplane and quotient bundle on the projective space $\PP(V_9/L)$. This implies that 
$$q^*\wedge^{j}\ke^*_{|\PP(V_9/L)}=(\wedge^{j-1}\kq^*\oplus \wedge^{j}\kq^*)(-j).$$
That this bundle has no cohomology in degree bigger than $j$ then follows directly from 
Bott's theorem. 
\end{proof}

Let $C_\omega=C_1\cup S\subset C$ denote the locus over which $p: q^* W_{\omega} \to C$ is not an
isomorphism, and $E \subset q^* W_\omega$ the exceptional locus $E:=p^{-1}(C_\omega)$, which is
a divisor. We denote by $\ko_E(h):=\ko_\pi(1)$ the relative hyperplane section.

\smallskip
Let $\tilde C_\omega\subset 
\Fl(1,3,V_9)$ be the variety of flags $U_1\subset U_3$ such that $\omega(U_1,U_3)=0$ and 
$\Omega(U_1,U_3,\bullet)=0$. In other words, $\tilde C_\omega$ is the zero-locus of the 
global section of the vector bundle $$\ke=\ke_1\oplus\ke_2=(\ku_1\wedge \ku_3)^*\oplus (\ku_1\wedge \ku_3\wedge V_9)^*$$ 
defined by $(\omega, \Omega)$. 
This bundle is globally generated of rank $2+13=15$, therefore $\tilde C_\omega$ is smooth
of dimension $20-15=5$. The projection to $\PP (V_9)$ gives a map $\eta : \tilde C_\omega\rightarrow C_\omega$,
which is bijective outside $S$. 
Over $U_1=\langle l\rangle\in S$, the kernel of $\Omega_l$ is five
dimensional and its intersection $U_4$ with $U_1^\perp$ is four dimensional.
The fiber of $\eta$ over $U_1$ is thus the set of three-dimensional spaces $U_3$ 
such that $U_1\subset U_3\subset U_4$,
hence a projective plane. 

We are going to show that the map $\eta: \tilde{C}_{\omega} \to C_\omega$ is the blow-up of $C_\omega$ along $S$,
and deduce that $C_\omega$ is smooth. This will require several steps. 

\begin{lemma}\label{van}
$\tilde C_\omega$ is irreducible and $h^{0,q}(\tilde C_\omega)=0$ for all $q>0$.
\end{lemma} 

\proof We resolve the structure sheaf of $\tilde C_\omega$ by the Koszul complex
\begin{equation}\label{koszulctilde}
0\lra \wedge^{15}\ke^*\lra\cdots \lra \ke^*\lra\ko_{\mathrm{Fl}}\lra \ko_{\tilde C_\omega}\lra 0.
\end{equation}
We will show that all the wedge powers $\wedge^{q}\ke^*$ are acyclic for $q>0$ and the claim will follow.
In order to check this acyclicity, we cannot apply the Bott-Borel-Weil theorem directly, because $\ke$ is not a completely reducible 
homogeneous vector bundle. In fact $\ke_1$ is irreducible but $\ke_2$ is not semisimple. Indeed, consider the
quotient bundles $\kq_2=\ku_3/\ku_1$ and $\kq_6=V_9/\ku_3$. Then $\ke_1^*=\ku_1\otimes \kq_2$ and there is an exact sequence 
$$ 0\lra \ke_3^*:=\ku_1\otimes \det(\kq_2)\lra \ke_2^*\lra \ke_4^*:=\ku_1\otimes \kq_2\otimes \kq_6\lra 0.$$
In order to prove that $H^q(\mathrm{Fl},\wedge^{q}\ke^*)=0$, it is enough to check that 
$$H^q(\mathrm{Fl},\wedge^{q_1}\ke_1^*\otimes\wedge^{q_3}\ke_3^*\otimes \wedge^{q_4}\ke_4^*)=
H^q(\mathrm{Fl},\wedge^{q_1}\kq_2\otimes\wedge^{q_3}\det(\kq_2)\otimes \wedge^{q_4}(\kq_2\otimes \kq_6)\otimes \ku_1^q)=0$$
when $q_1+q_3+q_4=q$. Note that $\ke_3$ is a line bundle, so we can suppose that $q_3\le 1$. 
By the Cauchy formula, we can decompose 
$$\wedge^{q_4}(\kq_2\otimes \kq_6)=\bigoplus_{a+b=q_4}S_{a,b}\kq_2\otimes S_{2^b1^{a-b}}\kq_6,$$
where $S_{a,b}$ and $S_{2^b1^{a-b}}$ are the Schur functors associated respectively with the partitions
$(a,b)$ (so that $a\ge b$) and $(2,\ldots ,2,1,\ldots ,1)$, with $b$ twos and $a-b$ ones (so that necessarily $a\le 6$). Tensoring by 
$\wedge^{q_1}\kq_2\otimes\wedge^{q_3}\det(\kq_2)\otimes \ku_1^q$, we get a direct sum of irreducible bundles
of the form
$$S_{2^b1^{a-b}}\kq_6\otimes S_{c,d}\kq_2\otimes \ku_1^q.$$
Now we are in position to apply the Bott-Borel-Weil theorem. Let $\rho=(8,\ldots ,2,1,0)$.  For the latter bundle not 
to be acyclic, we need that the sequence $$\sigma =(2,\ldots ,2,1,\ldots ,1,0,\ldots ,0,c,d,q)+\rho$$
admits no repetition. The seven leftmost terms of $\sigma$ give all the integers between $10$ and $3$, except 
$10-b$ and $9-a$. Since $S_{c,d}\kq_2$ is a direct factor of $S_{a,b}\kq_2\otimes \wedge^{q_1}\kq_2\otimes\wedge^{q_3}\det(\kq_2)$, 
we have $d\le c\le a+2\le 8$. So if $d\ge 2$, we need $c+2=10-b$ and $d+1=9-a$, that is $b+c=a+d=8$, and then all 
the integers between $10$ and $3$ appear in $\sigma$. So $q$ must be either bigger than $10$ or smaller than $2$.
But $c+d=a+b+q_1+2q_3$, hence $16=a+b+c+d=2q_4+q_1+2q_3=2q-q_1$. This yields $q=8+q_1/2$ with $0\le q_1\le 2$, which
gives a contradiction. 

So we need $d\le 1$, hence $b+q_3\le 1$. Then $q=a+b+q_1+q_3\le a+q_1+1\le 9$ since $a\le 6$ and 
$q_1\le 2$. If $q\ge 3$, the two integers $q$ and $c+2$ must coincide with $10-b$ and $9-a$. 
In particular $q+c+2=19-a-b$, that is $q_1+q_3+2a+2b+c=17$, and since $c\le a+q_3+q_1$ we get
$17\le 2(q_1+q_3+a+b)$, hence $9\le q_1+q_3+a+b\le q_1+a+1$. This is only possible for $a=6$, 
$q_1=2$, $b+q_3=1$, hence $q=9$. Since $\{q,c+2\}=\{10-b,9-a\}$ and $b\le 1$, we must have 
$q=6=9-a$ and $c+2=10-b$, hence $a=3$ and $c=8-b$. But then $c\ge 7$, and since necessarily 
$c\le a+2$, we get a contradiction. 

We are thus reduced to $q\le 2$, $d\le 1$ hence also $b\le 1$. Moreover, if $c>0$, we must have
$c=10-b$ or $9-a$.
But $c\le a+2\le 8$, so only $c=9-a$ is possible. Then $9-a\le a+2$ yields $a\ge 4$, 
and then $q\ge q_4=a+b\ge 4$, a contradiction. So finally $c=0$, hence also $d=0$, and since $c+d=q_1+2q_3+q_4$ we get $q_1=q_3=q_4=q=0$, as claimed.

\qed

\begin{lemma}\label{ampleness}
The line bundle $\km=\det(\kq_6)$ is ample on $\tilde C_\omega$.
\end{lemma} 

\proof Consider the projection $\psi : \tilde C_\omega\lra \Gr(3,V_9)$. It suffices to check that $\psi$
is finite on its image. Recall that $\tilde C_\omega$ is defined by the conditions that $\omega(\ku_1,\ku_3)=0$ and 
$\Omega(\ku_1,\ku_3,\bullet)=0$. For a fixed $U_3$, these are linear conditions on $U_1$, so if there is a non 
trivial fiber over $U_3$, there must exist a plane $U_2\subset U_3$ such that $\omega(U_2,U_3)=0$ and 
$\Omega(U_2,U_3,\bullet)=0$. This would give a point in the zero-locus of a general section of the vector bundle 
$(\ku_2\wedge \ku_3)^*\oplus (\ku_2\wedge \ku_3\wedge V_9)^*$ over the flag manifold $\Fl(2,3,V_9)$. But this is 
a vector bundle of rank $3+19=22$ over a flag manifold of dimension $20$, so this cannot happen. \qed 

\medskip
Then consider $\kl = \ku_1^*$ on $\mathrm{Fl}(1,3,V_9)$, the pullback of the hyperplane line bundle from $\PP (V_9)$.

\begin{lemma}\label{projnorm}
For any $m>0$, the restriction map $$H^0(\mathrm{Fl}(1,3,V_9),\kl^m)\lra  H^0(\tilde C_\omega,\kl_{|\tilde C_\omega}^m)$$ is surjective. 
Moreover it is an isomorphism for $m=1$. 
\end{lemma} 

\proof Again we use the Koszul complex (\ref{koszulctilde}) and Bott-Borel-Weil.\qed

\medskip
Now we are in position to apply \cite{aw}. By adjunction, the 
canonical bundle of $\tilde C_\omega$ is 
$$K_{\tilde C_\omega}=(4\kl-2\km)_{|\tilde C_\omega}.$$
By Lemma \ref{ampleness}, the line bundle $\km_{|\tilde C_\omega}$ is ample, so we 
can apply \cite[Theorem 4.1]{aw} to the pair $(X,L)=(\tilde C_\omega,\km_{|\tilde C_\omega})$,
with $r=2$. We claim that the {\it adjoint contraction morphism} defined by $K_X+2L$ is $\psi$. 
Indeed, $K_{\tilde C_\omega}+2\km_{|\tilde C_\omega}=4\kl_{|\tilde C_\omega}$, so by definition
this contraction morphism is the one defined by the linear systems $|4m\kl_{|\tilde C_\omega}|$
for $m>>1$.  But by Lemma \ref{projnorm}, this is the same morphism as the one defined by the 
linear system $|\kl_{|\tilde C_\omega}|$, which is indeed $\psi$. 

Since $\psi$ is birational with non trivial fibers isomorphic to $\PP^2$, \cite[Theorem 4.1(iii)]{aw}
applies and we conclude that:

\begin{prop}\label{blowupS}
$C_\omega$ is smooth and $\psi : \tilde C_\omega \lra C_\omega$ is the blow-up of $S$. 
\end{prop}

\begin{remark} Pushing the analysis a little further, one can deduce that $C_\omega$ has Picard rank one, since $\tilde C_\omega$ has Picard rank two. Indeed, since $h^{0,2}(\tilde C_\omega)=0$
by Lemma \ref{van}, we just need to prove that $h^{1,1}(\tilde C_\omega)=2$.
For this, it is enough to show that the maps 
$$H^1(\Omega_{\mathrm{Fl}})\lra H^1(\Omega_{{\mathrm{Fl}}|\tilde C_\omega})\lra H^1(\Omega_{\tilde C_\omega})$$
are both surjective. Using the Koszul complex as above, this follows from the vanishings 
$$H^{q+2}({\mathrm{Fl}}, \ke^*\otimes \wedge^q\ke^*)=H^{q+2}({\mathrm{Fl}}, \Omega_{\mathrm{Fl}}\otimes \wedge^{q+1}\ke^*)=0 \quad \forall q\ge 0,$$
which can be checked by applying Bott-Borel-Weil as above. 
\end{remark}

Now we will draw some consequences at the categorical level. Recall that the map $E\lra C_\omega$
has fibers $\PP^2$ over $S$ and fibers $\PP^1$ outside $S$. Moreover we denote by $F$ the preimage
of $S$. We will need two more lemmas.

\begin{lemma}\label{normalE}
Let $\kl$ and $\kd$ the pull-backs by $p$ and $q$ of the minimal ample line bundles on $\PP(V_9)$ and $\Gr(2,V_9)$, respectively. Then
$$N_{E/q^* W_\omega}=4\kl-\kd .$$
\end{lemma}

\begin{proof}
Inside $q^*W_\omega$, the divisor $E$ is defined as the set of pairs $(U_1,U_2)$ such that 
for $l\in U_1$ non zero, the kernel
of $\Omega_l$ is contained in $U_1^\perp$. Over $C$ the form $\Omega_l$ is degenerate, and outside 
$S$ its kernel $U_3$ is three dimensional. Note that we can choose linear forms $u_1,\ldots ,u_6$ 
such that $\Omega_l=u_1\wedge u_2+u_3\wedge u_4+u_5\wedge u_6$, and $U_3$ is then the intersection $u_1^\perp\cap\cdots \cap u_6^\perp$. So the decomposable form $\Omega_l\wedge \Omega_l\wedge \Omega_l = 6u_1\wedge u_2\wedge u_3\wedge u_4\wedge u_5\wedge u_6\in\wedge^6V_9^*$ represents $U_3$, and 
through the isomorphism $\wedge^6V_9^*\simeq  \wedge^3V_9$, this decomposable
form can be written as $p_1\wedge p_2\wedge p_3$ for $p_1, p_2, p_3$ some basis of $U_3$. 
Since $U_3\supset U_2\supset U_1$, we can write  
$p_1\wedge p_2\wedge p_3=l\wedge u_2\wedge u_3$ for some $u_2\in U_2$ and $u_3\in U_3$. Since $\omega(l,u_2)=0$ the contraction  
by the linear form $\omega(l,\bullet)$ gives $\omega(l,u_3)\, l\wedge u_2$, which vanishes if 
and only if $U_3$  is contained in $U_1^\perp$ (or $u_3=0$ if we are over $S$).
This means that over $q^*W_\omega$, 
$$\Omega_l\wedge \Omega_l\wedge \Omega_l\wedge\omega(l,\bullet)\in \ku_1^{-4}\otimes \det(\ku_2)$$ defines a natural section of $4\kl-\kd$, vanishing exactly along $E$. 
This implies the claim.
\end{proof}

Finally we compute the normal bundle of $F$ inside $E$. Recall that for $U_1\in S$, and $l$ a generator of the line $U_1$, the two-form $\Omega_l$ has a four-dimensional kernel mod $U_1$. This defines a rank five vector bundle $\ku_5$ on $S$, and a rank four bundle $\ku_4=\ku_5
\cap U_1^\perp$ (the latter intersection being everywhere transverse when $\omega$ and $\Omega $ are sufficiently
general). Moreover, $F$ is the total space of the fibration $\PP (\ku_4/\ku_1)$ over $S$. 

\begin{lemma}\label{normalF}
Consider the projective bundle $F=\PP (\ku_4/\ku_1)\rightarrow S$. Then 
the normal bundle of $F$ in $E$ is  dual to the tautological quotient 
bundle of the fibration.
\end{lemma}

\begin{proof}
Recall that $\tilde C_\omega\subset 
\Fl:=\Fl(1,3,V_9)$ was defined as the  variety of flags $U_1\subset U_3$ such that $\omega(U_1,U_3)=0$ 
and $\Omega(U_1,U_3,\bullet)=0$. Denote by $\Delta$ the exceptional divisor of the projection to 
$\PP (V_9)$, which by Proposition \ref{blowupS} is nothing else than the blow-up of $S$ in $C_\omega$. 

Let $\tilde E$ denote the total space of the projective bundle $\PP (\ku_3/\ku_1)$ over 
$\tilde C_\omega$, and $\tilde F$ its restriction to $\Delta$. By forgetting $U_3$, we define
a morphism from $\tilde E$ to $E$, that sends $\tilde F$ to $F$:
$$
\xymatrix{
E & \tilde E \ar[l]_\gamma \ar[r] & \tilde C_\omega  \ar[r] & C_\omega \\
F \ar@{^{(}->}[u]& \tilde F \ar@{^{(}->}[u] \ar[l]
\ar[r] & \Delta\ar[r]\ar@{^{(}->}[u] & S\ar@{^{(}->}[u].
}
$$

By construction, $\gamma$ is an isomorphism outside $F$, and a $\PP^1$-bundle over $F$. 
More precisely, $\tilde F$ is the total space of the projective bundle $\PP (\ku_4/\ku_2)$ 
over $F$. This readily implies that $\tilde E$ is just the blow-up of $F$ in $E$. In particular
the exceptional divisor of this blow-up, that is $\tilde F$, is the total space of the 
projectivized normal bundle $\PP (\kn_{F/E})$. We conclude that $\kn_{F/E}\simeq \ku_4/\ku_2\otimes M$,
for some line bundle $M$ on $F$. 

There remains to identify this line bundle $M$. Since the Picard group of $F$ is torsion free, 
it is enough to compare the determinants in the previous identity. 
First recall that the canonical bundle of $W$ is the restriction of $\det(\ku_2)^3$, hence that of $W_\omega$ is $\det(\ku_2)^2$. Taking determinants in the 
tangent short exact sequence 
$$0\lra Hom(\ku_1,\ku_2/\ku_1)\lra T_{q^*W_\omega}\lra q^*T_{W_\omega}\lra 0$$
we deduce that the canonical bundle of $q^*W_\omega$ is $K_{q^*W_\omega}=(\ku_1)^2\otimes \det (\ku_2)$. Then
Lemma \ref{normalE} implies that $K_E=(\ku_1)^{-2}\otimes \det (\ku_2)^2$. Second, since $F$ is 
the total space of the projective bundle $\PP (\ku_4/\ku_1)$ over $S$, we get  
$K_F=(\ku_1)^{-2}\otimes \det (\ku_2)^3\otimes \det (\ku_4)^{-1}$. We deduce that the
relative canonical bundle 
$$K_{F/E}=\det (\kn_{F/E})=\det (\ku_2)\otimes \det (\ku_4)^{-1}.$$
Therefore $M$ is also isomorphic to $\det (\ku_2)\otimes \det (\ku_4)^{-1}$, and we conclude
that 
$$\kn_{F/E}\simeq \ku_4/\ku_2\otimes \det (\ku_4/\ku_2)^*\simeq (\ku_4/\ku_2)^*$$
is dual to the tautological quotient bundle, as claimed. 
\end{proof}

 \medskip 
 Now Corollary \ref{cor:ourcases} applies and we get:

\begin{prop}\label{sodComega}
There is a fully faithful functor:
$$\Phi: \Db(S) \to \Db(E)$$
and a semiorthogonal decomposition
$$\Db(E) = \sod{\pi^*\Db(C_\omega)(-h),\Phi\Db(S),\pi^* \Db(C_\omega)}.$$
In particular, this decomposition yields a dual Lefschetz decomposition with respect to the line
bundle $\ko_E(h)$ by setting:
$$\cat{B}_0:=\sod{\Phi\Db(S),\pi^*\Db(C_\omega)} \subset \cat{B}_{1}:=\pi^*\Db(C_\omega).$$
\end{prop}

\begin{theorem}\label{weakly-crepant}
The category
$$\tilde{\cat{D}}:= \sod{j_*\pi^* \Db(C_\omega)}^\perp \subset \Db(q^*W_\omega)$$
is a weakly crepant categorical resolution of singularities of the Coble cubic $C$.
\end{theorem}

\begin{proof}
Since by Proposition \ref{coblerat} the Coble cubic $C$ has rational singularities, we are
in position to apply Theorem 1 of  \cite{kuz-lefschetz}. In order that the hypothesis of 
this Theorem are satisfied, we need to check that :
\begin{enumerate}
\item The conormal bundle  $\kn_{E/q^* W_\omega}^*\simeq \ko_E(h)$
(up to $\pi^*\Pic(C_\omega)$). Then the semiorthogonal decomposition of $\Db(E)$ from Proposition \ref{sodComega} is a Lefschetz decomposition with respect to the conormal bundle $\kn_{E/q^*W_\omega}^*$, and Kuznetsov's theorem ensures 
that $\tilde{\cat{D}}$ is a categorical resolution
of singularities of $C$. 
\item $C$ is Gorenstein, and its canonical bundle verifies $K_{q^*W_\omega}=p^*K_C + E$. Then
since obviously $\pi^* \Db(C_\omega) \subset\cat{B}_1$ (they are indeed equal!!), 
Kuznetsov's theorem ensures that the categorical resolution is weakly crepant.
\end{enumerate}
The first claim is an immediate consequence of Lemma \ref{normalE}. 
 The second claim readily follows: indeed 
$C$ is obviously Gorenstein, being a hypersurface, and its 
canonical bundle is $K_C=\ko_C(-6)$. Moreover, we computed in the proof of Lemma
\ref{normalF} that the canonical bundle of $q^*W_\omega$ is $-2\kl-\kd=(-6\kl)+(4\kl-\kd)$.
This concludes the proof.
\end{proof}
\noindent {\it Question}. The traditional construction of Coble cubics is in terms of 
vector bundles on genus two curves, see \cite{beauville-coble}. Is there a modular 
interpretation of our constructions? 

\begin{remark}
Note that the above diagram allows us to obtain the following equation in the Grothendieck ring
$K_0(\mathrm{Var}(\CC))$:
$$[q^* W_\omega] = [C] + \LL [C_\omega] + \LL^2 [S].$$
The subcategory $\tilde{\cat{D}}$ being the orthogonal to one copy of $\Db(C_\omega)$ confirms
the expectations from the previous construction, that is that the resolution of singularities of
$C$ would be written as $[C] + \LL^2 [S]$ in the Grothendieck ring (if it was a variety!). 

Moreover, assuming conjecture \ref{conj:deco-for-T29}, one gets a semiorthogonal decomposition
for the hyperplane section $W_\omega$ of $W$:
$$\Db(W_\omega) = \sod{\cat{A}_\omega,\Db(\Gamma),\ko,\ku^*,\Db(\Gamma),\ko(1),\ku^*(1)},$$
for some category $\cat{A}_\omega$. In particular the $\PP^1$-bundle $q^* W_\omega$ would admit
a semiorthogonal decomposition by $4$ copies of $\Db(\Gamma)$, $8$ exceptional objects, and $2$ copies
of $\cat{A}_\omega$. On the other hand, the resolution of singularities $\tilde{\cat{D}}$ is the orthogonal
complement of a copy of $\Db(C_\omega)$ in $\Db(q^*W_\omega)$. The combination of conjectures
\ref{conj:deco-for-T29} and \ref{conj:resol-Coble} lets one expect that $\Db(C_\omega)$
admits a semiorthogonal decomposition by $2$ copies of $\cat{A}_\omega$, one copy of $\Db(\Gamma)$
and $2$ exceptional objects.
\end{remark}

\subsection{Linear section of $\Gr(3,11)$ and a non-geometrical 3CY category}
Finally, we will briefly consider the hyperplane section $Y \subset \Gr(3,V_{11})$, which is a 3-FCY and is a derived pure 3-CY Fano variety. In fact, $Y$ has a semiorthogonal
decomposition
$$\Db(Y)= \sod{\cat{A},E_1,\ldots,E_{150}},$$
where $\cat{A}$ is a 3CY category \cite{kuz-fractional} and $E_1,\ldots,E_{150}$ are 
exceptional objects. Moreover, $Y$ is also of 3CY type. One can proceed with correspondences
induced by jumps and projections to spread the Hodge structure and (conjecturally) the category $\cat{A}$
in other varieties. A quick analysis of the possible target varieties easily leads to show that there is
no geometrical Calabi-Yau threefold in the picture. On the other hand, one can also show that 
for numerical reasons, the category $\cat{A}$ cannot be geometrical.

\begin{prop}
There is no projective Calabi-Yau threefold $X$ such that $\cat{A} \simeq \Db(X)$.
\end{prop}
\begin{proof}
First of all, thanks to \cite{kuz-hochsch}, and the above semiorthogonal decomposition, we have
$$HH_0(Y)=HH_0(\cat{A}) \oplus \QQ^{\oplus 150},$$
where the second component is given by the exceptional objects $E_1,\ldots,E_{150}$.
Moreover, $HH_i(Y)=HH_i(\cat{A})$ for $i \neq 0$.

Calculating the Hodge numbers, we get that the only nontrivial noncentral Hodge numbers
of $Y$ give a middle cohomology of 3CY type as follows:
$$1 \,\, 44 \,\, 44 \,\, 1,$$
so that $\dim HH_1(\cat{A})=44$, $\dim HH_2(\cat{A})=0$, and $\dim HH_3(\cat{A})=1$.
Using that the Euler characteristic is the alternate sum of the dimensions of the Hochschild homology groups, we get
$$\chi(Y)=\dim HH_0(A) + 150 - 90.$$
The Euler characteristic of $Y$ can be calculated to be $62$,
hence we would have $\dim HH_0(\cat{A})=2$. But
if $X$ is a smooth projective Calabi-Yau threefold, then $HH_0(X) \geq 4$, and this concludes the 
proof.
\end{proof}

\subsection{A cascade of examples with multiple CY structure}
As calculated in Theorem \ref{thm:sect-of-G}, a smooth hyperplane
section of $\Gr(k,V_n)$ is a Fano of $r$-CY type (of derived $r$-CY type if $k$ and $n$ are
coprime \cite{kuz-fractional}), where $r=k(n-k)+1-2n$, with $n>3k$ and $k>2$. In particular,
the only possible values for which $r=2$ are $n=10$ and $k=3$, the case treated above.
However, the above correspondences, notably those induced by jumps, can be applied in this more
general case to produce varieties with multiple $r$-CY structure, as follows.

Let $Y \subset \Gr(k,V_n)$ a hyperplane section given by a $k$-form $\Omega$ on $V_n$.
Then we can define the {\it first $k$-alternate congurence Grassmannian} to be the variety
$Z \subset \Gr(k-1,V_n)$ of those $k-1$ planes $U \subset V_n$ such that the form
$\Omega(U,\bullet)$ is degenerate. Such $Z$ is a locus of a general section of
$\kq^*(1)$ and is hence smooth of dimension $n-k+1$, and has canonical bundle
$\omega_Z \simeq \ko_Z(-k)$. The $(k,k-1)$ jump on $V_n$ allows then us to calculate the
Hodge numbers of $Z$ and obtain:
\begin{itemize}
\item The Picard rank of $Z$ is $1$.
\item $Z$ is Fano of $r$-CY type, namely $H^j(Z,\CC)$ is $r$-CY for $j=n-2i$ 
and $i=0,\ldots,k-1$, while $H^{p,q}(Z)=0$ if $p\neq q$ for $p+q > 2n$ and
$p+q < 2n-2k+2$.
\end{itemize}

Similarly, if $\cat{A} \subset \Db(Y)$ is the $r$-CY category orthogonal to an
exceptional collection (such $\cat{A}$ exists for $k$ and $n$ coprime) one should
expect $\Db(X)$ to admit a decomposition with $k$ copies of $\cat{A}$ and exceptional
object. Similarly to the cases $n=9,10$ and $k=3$, since the canonical bundle of $Z$
is $\ko(-k)$, we suspect to have a Lefschetz decomposition, but not necessarily
rectangular. Some numerology:

\begin{itemize}
\item The full exceptional collection of $\Gr(k,V_n)$ has ${n \choose k}$ objects,
that can be organized in a rectangular Lefschetz decomposition with $n$ components, each made
hence of $\frac{(n-1)!}{(n-k)!k!}$ objects.
\item $\cat{A}$ is the orthogonal complement in $\Db(Y)$ of an exceptional collection
made of $n-1$ components of the Lefschetz decomposition above. Hence the exceptional collection
on $Y$ has length $\frac{(n-1)(n-1)!}{(n-k)!k!}$.
\item The $\PP^{k-1}$ bundle $q^*Y \to Y$ has then $k$ copies of $\cat{A}$ and an exceptional
collection of length $a=\frac{(n-1)(n-1)!}{(n-k)!(k-1)!}$.
\item The Grassmannian $\Gr(k-1,V_n)$ has a full exceptional collection of length
${n \choose k-1}=\frac{n!}{(n-k+1)!(k-1)!}$.
\item The map $p:q^*Y \to \Gr(k-1,V_n)$ is generically a $\PP^{n-k-1}$-bundle, so
the orthogonal to $\Db(Z)$ in there is given by $n-k-1$ copies of the Grassmannian. It
follows that we have $b=\frac{(n-k)n!}{(n-k+1)!(k-1)!}$ exceptional objects orthogonal
to $\Db(Z)$.
\end{itemize}

From the above, we can then expect to have $\Db(Z)$ generated by $k$ copies of $\cat{A}$
and a number of exceptional objects that we can calculate as
$$\begin{array}{c}
a-b=\frac{(n-1)!}{(n-k+1)!(k-1)!}((n-1)(n-k+1)-n(n-k))=\\
\\
=\frac{(n-1)!}{(n-k+1)!(k-1)!}(k-1)= \frac{(n-1)!}{(n-k+1)!(k-2)!} = {n-1 \choose k-2}
\end{array}$$

\appendix

\section{A decomposition of the Hodge structure}

Let $X$ be a smooth projective variety, $Z \subset X$ a smooth codimension $c$ subvariety
and $\sigma: Y \to X$ be the blow-up of $X$ along $Z$ with exceptional divisor $j: E \hookrightarrow X$.
In particular, $p: E \to Z$ is a projective bundle of relative dimension $c-1$, with relative
ample line bundle $\ko_E(H)=\ko_Y(-E)_{|E}$.
In this case, it is well known that
we can decompose both the Hodge structure $H^j(Y,\CC)$ (see, e.g. \cite[7.3.3]{voisin-book})
and the derived category $\Db(Y)$ (see \cite{orlovprojbund}) in terms
of their counterparts on $X$ and $Z$. 

We generalize these results to the following situation: $\pi:Y \to X$ is a proper map between smooth projective
varieties, and there is a smooth subvariety $\iota: Z \subset X$ of codimension $c \geq 2$, and integers $n < m < n+c$
such that the map $\pi$ is a 
$\PP^n$-bundle over $X \setminus Z$ and a $\PP^m$-bundle over $Z$. That is, there is a smooth projective subvariety
$j:F \subset Y$ of codimension $d=c+n-m$, a commutative diagram
\begin{equation}\label{eq:diagram-projective-bundles}
\xymatrix{
F \ar@{^{(}->}[r]^j \ar[d]_p & Y \ar[d]^\pi \\
Z \ar@{^{(}->}[r]^\iota & X,}
\end{equation}
and a locally free sheaf $\kf$ of rank $m+1$ on $Z$ such that $p:F \simeq \PP_Z(\kf) \to Z$.
We denote by $\ko_F(H)$ the relative ample bundle of $p$ and we assume that there is a line
bundle $\ko_Y(H)$ such that $\ko_Y(H)_{\vert F} \simeq \ko_F(H)$,

We denote by $h$ and $h_F$ the first Chern classes
of $\ko_Y(H)$ and $\ko_F(H)$ respectively.

We start with the Hodge-theoretical result. The following Proposition is probably well-known to the
experts.

\begin{prop}\label{prop:main-Hdg-deco}
In the configuration above, there is an isomorphism of integral Hodge structures:

$$\bigoplus_{i=0}^n H^{j-2i}(X,\CC)(i) \oplus \bigoplus_{i=0}^{m-n-1} H^{j-2i-2d}(Z,\CC)(d+i) \simeq H^j(Y,\CC) $$

given by the map
$$\phi:=\sum_{i=0}^n h^i \circ \pi^* + \sum_{i=0}^{m-n-1} j_* \circ h_F^i \circ p^*.$$
\end{prop}
\begin{proof}
The proof follows closely the proof of the Hodge decomposition of a blow-up, see, e.g. \cite[7.3.3]{voisin-book}.
First of all, the morphism $\phi$ is a morphism of Hodge structures, as a composition of morphisms of
Hodge structures. We are left to prove that $\phi$ gives an isomorphism of the underlying $\ZZ$-modules.

Let $U \subset X$ be the open subset $U:=X \setminus Z$. Then by assumption, $Y_U:=\pi^{-1} U$
is a $\PP^n$-bundle over $U$. Hence, the integral cohomology $H^*(Y_U,\ZZ)$ is a free module over the ring
$H^*(U,\ZZ)$ with basis $1,\ldots,h^n$.
On the other hand, $F \to Z$ is a $\PP^m$-bundle, so that the integral cohomology $H^*(F,\ZZ)$
is a free module over the ring $H^*(Z,\ZZ)$ with basis $1,h_F,\ldots,h_F^m$.

Note that, by excision and the Thom isomorphism, we can identify the integral cohomologies of the pairs $(X,U)$
and $(Y,Y_U)$ as follows:
\begin{equation}\label{eq:thom}
H^{j-1} (X,U) \simeq H^{j-2c}(Z), \,\,\,\,\,\, H^{j-1}(Y,Y_U) \simeq H^{j-2d}(F).
\end{equation}

Given an integer $j$, we draw the following diagram obtained from the long exact sequences for the
relative cohomology of the pairs $(X,U)$ and $(Y,Y_U)$:
\begin{equation}
\xymatrix{
& \bigoplus_{i=0}^n H^{j-2c-2i}(Z) \ar[d]^{\simeq} \ar@/_5pc/@<-2ex>[ddd]^{\overline{\alpha}} & &\\
\bigoplus_{i=0}^n H^{j-1-2i}(U) \ar[d]_{\sum h^i \circ \pi_U^* } \ar[r]|\hole & \bigoplus_{i=0}^n H^{j-1-2i} (X,U) \ar[d]^{\sum  h^i \circ \pi_{(X,U)}^*} \ar[r] & \bigoplus_{i=0}^n H^{j-2i}(X) \ar[d]^{\sum h^i \circ \pi^*} \ar[r] &
\bigoplus_{i=0}^n H^{j-2i}(U) \ar[d]^{\sum h^i \circ \pi_U^*} \\
H^{j-1}(Y_U) \ar[r]|\hole & H^{j-1} (Y,Y_U) \ar[d]^{\simeq} \ar[r] & H^j(Y) \ar[r] & H^j(U)\\
& H^{j-2d}(F).& &
}
\end{equation}
In particular, there is a surjective map:
$$\overline{\beta}:(\sum h^i \circ \pi^*, j_*): \bigoplus_{i=0}^n H^{j-2i}(X) \oplus H^{j-2d}(F) \to H^j(Y).$$

In order to understand the kernel of $\overline{\beta}$, we consider the composed map $\overline{\alpha}$.
As in \cite[7.3.3]{voisin-book}, we
can see first that $\overline{\alpha}$ is given by $h_F^{i+m-n}\circ \pi^*$ on each component $H^{j-2c-2i}(Z)$,
which is then mapped to $H^{j-2d}(F)$ since $d=c+n-m$.
We end up with the map:
$$(h_F^{m-n}\circ p^*,\ldots,h^m \circ p^*): \bigoplus_{i=0}^n H^{j-2c-2i}(Z) \longrightarrow H^{j-2d}(F),$$
which is injective since $F \to Z$ is a projective bundle and $m > n$. On the other hand the left most term is equal to $\bigoplus_{i=m-n}^m H^{j-2d-2i}(Z)$, since $d=c+n-m$.

Then we conclude as in \cite[7.3.3]{voisin-book}. 
\end{proof}

\section{A semiorthogonal decomposition}\label{sect:appendix-A}

We keep the notations of the previous section, in particular from diagram \eqref{eq:diagram-projective-bundles}.
Let us assume moreover that $d>1$, that is, that $F$ is not a divisor in $Y$, and that the relative Picard
group $\Pic(Y/X)$ is free of rank 1 and generated by $\ko_Y(H)$. In particular, since $Y \to X$ is a $\PP^n$-bundle
in codimension 1 (on $Y$), we have the relative anticanonical bundle $\omega^*_{Y/X} \simeq \ko_Y((n+1)H)$,
and there is then a line bundle $L$ on $X$ such that $\omega^*_Y \simeq \pi^* L\otimes \ko_Y((n+1)H)$.
On the other hand, $p: F \to Z$ is a $\PP^m$-bundle, so that there exists a line bundle $M$ on $Z$
such that $\omega_F^* \simeq p^* M \otimes \ko_F((m+1)H)$. We finally note that, letting $M':=M^* \otimes \iota^* L$ in $\Pic(Z)$, 
the relative canonical bundle of the embedding $j$ is given by:
$$\omega_j = \omega_F \otimes j^*\omega_Y^* = p^*M' \otimes \ko_F((n-m)H).$$

We need the following additional conditions:

\begin{enumerate}
\item[C1)] If $F_z \simeq \PP^m$ is a fiber over a point $z$ of $Z$, 
the bundle $\bigwedge^s \kn_{F_z/Y}$ is acyclic for $s=0,\ldots,\dim Z$
\item[C2)] If $m > n+1$, the bundle $\bigwedge^s \kn^*_{F/Y}$ is left orthogonal to the categories
$p^* \Db(Z) \otimes \ko(-kH)$ for $k=1,\ldots,m-n-1$ and all $s$.
\end{enumerate}

We define the functors $\Phi_l:\Db(Z) \to \Db(F)$ by the formula
$\Phi_l(A)= j_* (p^* A \otimes \ko(lH))$.
The next Proposition is probably well-known to the experts, and holds probably with less
restrictive assumptions. The assumption C1) and C2) are indeed of rather technical nature:
we need C1) to show that $\Phi_l$ is fully faithful using the Bondal-Orlov
criterion (step 2 of the proof), and we need C2) to show that the collection of subcategories
$\Phi_l \Db(Z),\ldots,\Phi_{l+m-n}\Db(Z)$ is semiorthogonal.

\begin{prop}\label{prop:main-sod}
In the configuration above, if C1) holds, $\Phi_l$ is fully faithful for any integer $l$.
If moreover C2) also holds, there is a semiorthogonal decomposition:
$$\Db(Y)=\sod{\Phi_{n-m}\Db(Z),\ldots,\Phi_{-1}\Db(Z),\pi^*\Db(X),\ldots,\pi^*\Db(X)\otimes \ko_Y(nH)}.$$
\end{prop}

Before proceeding with the proof, we remark that a generalization of Orlov's
blow-up formula already appeared in \cite{Jiang-Leung}, in a slightly different context. There, the
case of the cokernel $G$ of a map $E \to F$ between two vector bundles on a variety $X$ with degeneracy locus $Z$
is considered. In such a case, setting $Y=\PP(G)$ we would have, in our notations, $m=n+1$, but only generically
along $Z$: the case $m=n+1$ of the above result coincide with the one from \cite{Jiang-Leung}
only if $Z$ is smooth. We finally would like to mention that the proof in \cite{Jiang-Leung} is based
on Homological Projective Duality and hence is very different from the proof we are giving here.

\begin{proof}
{\bf Step 1.}
First of all, for any integer $k$, the functor $\pi^* \otimes \ko_Y(kH)$ is fully faithful since it
is the composition of the fully faithful functor $\pi^*$ with the autoequivalence given by the tensor
product with the line bundle $\ko_Y(kH)$. Secondly, the semiorthogonality of the
sequence
$$\{ \pi^*\Db(X),\ldots,\pi^*\Db(X)\otimes \ko_Y(nH) \}$$
follows by 
relative Kodaira vanishing and the fact that the relative anticanonical bundle is $\ko_Y((n+1)H)$.

\medskip

{\bf Step 2.}
Now we check that the functor $\Phi_l : \Db(Z) \to \Db(Y)$ is fully faithful for any 
integer $l$.
In order to do that, we can proceed as in the proof of \cite[Prop. 11.16]{Huybook}.
First of all (see \cite[Prop. 11.8]{Huybook}), we have the following isomorphism
$${\mathcal{E}}xt_Y^k (j_*\ko_F,j_*\ko_F) \simeq \bigwedge^k \kn_{F/Y}.$$
The functor $\Phi_l$ is a Fourier--Mukai functor with kernel $\ko_F(lH)$, seen as an object of
$\Db(Z \times Y)$. Then it is enough to check the Bondal-Orlov equivalence criterion for Fourier--Mukai
functors \cite{bondal-orlov-archive}. First of all, if $z_1$ and $z_2$ are different points of $Z$,
their images via $\Phi_l$ have disjoint supports and hence there is no nontrivial ext between them.
There remains to show that for any
point $z$ of $Z$
$$\Ext^i_Y(\ko_{F_z}(lH),\ko_{F_z}(lH)) = \Ext^i_Y(\ko_{F_z},\ko_{F_z})$$
vanishes for $i<0$ and $i>\dim Z$ and is one-dimensional for $i=0$, where $F_z \simeq \PP^m$
is the fiber of $p$ over the point $z$. We follow \cite[Prop. 11.16]{Huybook}, and use the local-to-global
spectral sequence for the Ext groups, which, using
${\mathcal{E}}xt_Y^k (j_*\ko_{F_z},j_*\ko_{F_z}) \simeq \bigwedge^k \kn_{F_z/Y}$ reads:
$$E^{r,s}_2=H^r(F_z,\bigwedge^s \kn_{F_z/Y}) \Longrightarrow \Ext_Y^{r+s}(\ko_{F_z},\ko_{F_z}).$$
The bundle $\kn_{F_z/Y}$ can be calculated via the nested sequence:
$$0 \longrightarrow \kn_{F_z/F} \longrightarrow \kn_{F_z/Y} \longrightarrow \kn_{F/Y \vert F_z} \longrightarrow 0.$$
The required vanishings follow then from assumption C1).
\medskip

{\bf Step 3.}
Now we show that $\{\Phi_l \Db(Z),\ldots,\Phi_{l+m-n}\Db(Z) \}$ is a semiorthogonal collection in $\Db(Y)$
for any integer $l$. This step is needed only if $m > n+1$.

For $A$ and $B$ objects of $\Db(Z)$, we need to calculate:
\begin{equation*}\label{eq:want-to-vanish1}
\Hom_Y(j_*(p^* A \otimes \ko_F((l+k)H)), j_*(p^* B \otimes \ko_F(lH))=
\Hom_F(j^* j_* p^* A, p^* B \otimes \ko_F((-k)H))),
\end{equation*}
where the equality follows by adjunction. We want to show that the latter vanishes 
for $k=1, \ldots, m-n-1$.
In order to perform this calculation, we use the following exact
sequence (see \cite[Rmk. 3.7]{Huybook}):
$$E_2^{r,s} = \Ext^r (\kh^{-s}(C),D) \Longrightarrow \Ext^{r+s}(C,D),$$
for $C,D$ objects of $\Db(F)$.
Moreover, if $C$ is an object of $\Db(F)$, we have (see \cite[Cor. 11.2]{Huybook})
$$\kh^{-s} (j_* j^* C) = \bigoplus_{u-t=s} \wedge^t \kn_{F/Y}^* \otimes \kh^u(C).$$
Hence the claim will follow if we can show that for $l'$ in the above range, we have
\begin{equation}\label{semiorth4}
\Ext^r(\wedge^t \kn_{F/Y}^* \otimes p^* \kh^u (A),p^* B \otimes \ko_F(-kH))=0
\end{equation}
for any $r,t,u$ and $k=1,\ldots,n-m-1$. Indeed, plugging these trivial values into the above
exact sequence will give the required vanishings. But, the vanishings \eqref{semiorth4} are
a direct consequence of assumption C2).

\medskip

{\bf Step 4.}
Now we check that $\Phi_l \Db(Z)$ is left orthogonal to $\pi^* \Db(X) \otimes \ko_Y(rH)$ for all
$l,r$ such that $0 < r-l < m+1$, and therefore construct a semiorthogonal set of subcategories.

Let $A$ be in  $\Db(X)$,
and for any $B$ in $\Db(Z)$. We have:
\begin{equation}\label{eq:2nd-orthogonality}
 \Hom_Y(\pi^* A \otimes \ko(rH),j_* (p^*B \otimes \ko(lH)))= \Hom_F(p^* \iota^* A , p^*B \otimes \ko((l-r)H))= 0,
 \end{equation}
where we first use adjunction and the fact that $p \circ \iota = j \circ \pi$.
The claim follows again by the relative Kodaira vanishing for
the projective bundle $p: F \to Z$.

So, consider the subcategories $\{ \pi^* \Db(X),\ldots, \pi^* \Db(X)\otimes \ko_Y(nH) \}$. Then
$\Phi_l \Db(Z)$ is left orthogonal to all these categories if
$n-m \leq l \leq -1$.
\medskip

Using the hypothesis $d \geq n$ and combining Step 3 and 4,
we end up with the following subcategory of $\Db(Y)$:
$$\cat{T}=\sod{\Phi_{n-m}\Db(Z),\ldots,\Phi_{-1}\Db(Z),\pi^*\Db(X),\ldots,\pi^*\Db(X)\otimes \ko_Y(nH)}$$

{\bf Step 5.}
We want to show that $\cat{T}=\Db(Y)$. We will prove that $\cat{T}^\perp=0$.

So let $A$ be a non zero object of $\Db(Y)$ such that:
\begin{equation}\label{eq:orthogonal-to-PHIl}
\Hom_Y(j_*(p^* B \otimes \ko(lH)),A)=0
\end{equation}
for all $B$ in $\Db(Z)$ and for $l=n-m,\ldots,-1$. That is, $A$ is right orthogonal to 
$$\sod{\Phi_{n-m}\Db(Z),\ldots,\Phi_{-1}\Db(Z)}.$$
Recall that by Grothendieck-Verdier duality $j^!A=j^*A \otimes \omega_j [d]$ (see, e.g., \cite[Cor. 3.38]{Huybook}) and that 
$\omega_j = p^*M' \otimes \ko_F((n-m)H)$, for some
line bundle $M$ in $\Db(Z)$.
We deduce:
$$\Hom_F(p^*B \otimes \ko((l+m-n)H),j^*A)= 0$$
for all $B$ in $\Db(Z)$ and $0 \leq l+m-n \leq m-n-1$. Considering the semiorthogonal decomposition:
$$\Db(F) =\sod{p^*\Db(Z)\otimes \ko(-n-1),\ldots,p^*\Db(Z)\otimes \ko(m-n-1)},$$
we deduce that $j^*A$ belongs to the category
$$\sod{p^*\Db(Z)\otimes \ko(-n-1),\ldots,p^*\Db(Z)\otimes \ko(-1)}$$
and is in particular canonically filtered by objects $p^*C_{-s} \otimes \ko(-sH)$ for
$C_{-s}$ in $\Db(Z)$ and $1 \leq s \leq n+1$.

Now let us assume that $A$ is orthogonal to $\{\pi^*\Db(X),\ldots,\pi^*\Db(X)\otimes \ko_Y(nH)\}$.
First of all, this implies that $j^*A$ is nontrivial. Indeed, if $j^*A = 0$, then the support of $A$
is concentrated outside $F$, and then $A$ belongs to
the category
$$\sod{\pi^*\Db(X),\ldots,\pi^*\Db(X)\otimes \ko_Y(nH)}$$
since $Y \setminus F$ is a $\PP^n$-bundle over $X \setminus Z$.

Secondly, for any $B$ in $\Db(X)$ and any $t$ such that $0 \leq t \leq n$, we have:
$$0=\Hom_Y(\pi^* B \otimes \ko(tH),A)=\Hom_Y(\pi^* B \otimes \ko(tH),A\otimes \omega_Y \otimes \omega_Y^*).$$
Now apply Serre duality and recall that $\omega_Y^* = \ko_Y((n+1)H) \otimes \pi^* L$ for some $L$ in $\Pic(X)$ to obtain that
$$\Hom_Y(A \otimes \ko_Y(n+1-t),\pi^* B) = 0$$
for any $B$ in $\Db(X)$ and any $t$ in $\{0,...,n\}$, that is $r:=n+1-t$ ranges from $1$ to $n+1$.

Now let $A$ in $\cat{T}^\perp$. By the above consdierations, for any $1 \leq r \leq n+1$ and for any $B$
in $\Db(X)$, we have 
\begin{equation}\label{eq:another-necessary-vanishing}
\Hom_Y(A\otimes\ko_Y(r),\pi^*B)= 0
\end{equation}
and $j^*A$  is nontrivial and canonically filtered by objects $D_{-s}:=p^*C_{-s} \otimes \ko(-s H)$ for
$C_{-s}$ in $\Db(Z)$ and $1 \leq s \leq n+1$, as follows:
\begin{equation}\label{eq:the-filtration-for-A}
\xymatrix{
0 = T_{-1} \ar[r]^{\phi_{-1}} & T_{-2} \ar[r]^{\phi_{-2}} & \cdots \ar[r]^{\phi_{-n}} & T_{-n-1} \ar[r]^{\phi_{-n-1}} & j^*A}
\end{equation}
with $\mathrm{cone}(\phi_{-s}) = D_{-s}$.
In particular, there must exist an $s$ such that
$D_{-s}$, and therefore also $C_{-s}$, are nontrivial. The following Lemma will give a contradiction to $A \neq 0$.

\begin{lemma}
Let $s$ be such that $C_{-t}=0$ for any $t < s$, and $C_{-s} \neq 0$. 
Then there exists a point $z$ of $Z$ such that $\Hom_Y(A \otimes \ko(sH),\pi^*k(z)) \neq 0$.
\end{lemma}

\begin{proof}
First notice that by our assumption, the above filtration \eqref{eq:the-filtration-for-A}
can be simplified to
\begin{equation}\label{eq:the-2nd-filtration-for-A}
\xymatrix{
0 = T_{-s} \ar[r]^{\phi_{-s}} & T_{-s-1} \ar[r]^{\phi_{-s-1}} & \cdots \ar[r]^{\phi_{-n}}  & T_{-n-1} \ar[r]^{\phi_{-n-1}} & j^*A}
\end{equation}
Indeed, our assumption can be rephrased by asking that $j^*A$ belongs to the subcategory
$$\sod{p^* \Db(Z) \otimes \ko_F((-n-1)H),\ldots,p^* \Db(Z) \ko_F(-sH)}.$$
Now we proceed as in the proof of \cite[Prop. 11.18]{Huybook}, part iii). We will use the following spectral sequence:
\begin{equation}\label{eq:the-spect-seq-in-the-lemma}
E_2^{u,-v} = \Hom_Y(A \otimes \ko_Y(sH),\kh^{-v}( \pi^* k(z))[u]) \Longrightarrow \Hom_Y(A\otimes \ko_Y(sH), \pi^*k(z)[u-v]).
\end{equation}

Notice that (see e.g. \cite[Prop. 11.12]{Huybook}) $\kh^{-v}( \pi^* k(z)) \simeq j_*\Omega_{F_z}^v(v)$ and recall that the fiber
$F_z \simeq \PP^m$ is a projective space of dimension $m$. Now:
$$\begin{array}{rcl}
\Hom_Y(A \otimes \ko_Y(sH),\kh^{-v}( \pi^* k(z))[u]) & = & \Hom_Y(A \otimes \ko_Y(sH),j_*\Omega_{F_z}^v(v)[u])\\
 & =& \Hom_F(j^* A, \Omega^v(v-s)[u]),
\end{array}$$
by adjunction. So we need to calculate the last morphism space. We appeal to the filtration \eqref{eq:the-2nd-filtration-for-A}:
remark that, for $1 \leq t < s$, we have:
$$\begin{array}{rcl}
\Hom_F(D_{-t}, \Omega^v(v-s)[u]) & = & \Hom_F(p^* C_{-t}, \Omega^v(v-s+t)[u]) \\
 & =&  \Hom_Z(C_{-t}, p_*\Omega^v(v-s+t)[u])=0\end{array}$$
for all $u$ and $v$, since $-m < t-s < 0$ for $t$ in $\{1,\ldots,s-1\}$.

Plugging this into the exact triangles for the filtration \eqref{eq:the-2nd-filtration-for-A}, we obtain:
$$\Hom_F(j^* A, \Omega^v(v-s)[u])= \Hom_Z(C_{-s},p_*\Omega^v(v)[u])$$
and we conclude as in \cite[Prop. 11.18]{Huybook}.
\end{proof}

The proof is concluded since we have shown that an object $A$ which is orthogonal to
$$\sod{\Phi_{n-m}\Db(Z),\ldots,\Phi_{-1}\Db(Z),\pi^*\Db(X),\ldots,\pi^*\Db(X)\otimes \ko_Y(nH)}$$
in $\Db(Y)$ is trivial. 
\end{proof}

\subsection{Special cases}

We detail here two special cases where Proposition \ref{prop:main-sod} applies, that is,
where conditions C1) and C2) are satisfied. We denote by $\kr$ the tautological (relative)
quotient of the $\PP^m$-bundle $F \to Z$.

\begin{cor}\label{cor:ourcases}
Let $m=n+1$ and $\kn_{F/Y}=\kr^* \otimes p^* L$ for some line bundle $L$ on $Z$. Then there is a
semiorthogonal decomposition:
$$\Db(Y)=\sod{\Phi_{-1}\Db(Z),\pi^*\Db(X),\ldots,\pi^*\Db(X)\otimes \ko_Y(nH)}.$$
\end{cor}

\begin{proof}
Since $m=n+1$, we only need to check condition C1). But notice that under the assumptions,
using the nested sequence:
\begin{equation}\label{eq:nested-for-N}
0 \longrightarrow \kn_{F_z/F} \longrightarrow \kn_{F_z/Y} \longrightarrow \kn_{F/Y} \longrightarrow 0,
\end{equation}
we deduce that $\kn_{F_z/F} \simeq \kr^*_{\PP^m} \oplus \ko_{\PP^m}^{\oplus \dim Z}$, and
condition C1) follows.
\end{proof}

\begin{cor}
Assume $\kn_{F/Y} = \ko(-H) \otimes p^* \ke$, for some vector bundle $\ke$ on $Z$. This
holds in particular if $\ke$ is the restriction of a vector bundle on $X$ and $F$ is the
zero locus of a section of the above bundle. If $d \geq n$, there is a semiorthogonal
decomposition
$$\Db(Y)=\sod{\Phi_{n-m}\Db(Z),\ldots,\Phi_{-1}\Db(Z),\pi^*\Db(X),\ldots,\pi^*\Db(X)\otimes \ko_Y(nH)}.$$
\end{cor}

\begin{proof}
We need to check conditions C1) and C2).
Using the nested sequence \eqref{eq:nested-for-N}, we obtain that 
$\kn_{F_z/Y} \simeq \ko_{\PP^m}^{\oplus \dim Z} \oplus \ko_{\PP^m}(-1)^{\oplus d}$, and
C1) follows.

To check C2), note that $\bigwedge^s \kn_{F/Y}^*$ is trivial for $t<0$ and for $t>d$, and otherwise
$\bigwedge^s \kn_{F/Y}^*=p^*M_s \otimes \ko_F(sH)$ for some $M$ in $\Db(Z)$.
Moving $s$ from $0$ to $d-1$, the latter are all left orthogonal to $p^*\Db(Z)\otimes\ko(-kH)$ for
$k=1,\ldots,m-d-1$. Condition C2) follows then from our assumption $d \geq n$.
\end{proof}


\begin{thebibliography}{99}


\bibitem{at12}
N. Addington and R. Thomas, \emph{Hodge theory and derived categories
of cubic fourfolds}, Duke Math. J. {\bf 163} (2014), no.~10, 1885--1927.

\bibitem{aw} 
M. Andreatta, J. Wiśniewski, 
\emph{A note on nonvanishing and applications},
Duke Math. J. 72 (1993), no. 3, 739–755.

\bibitem{auel-berna-frg}
A. Auel, M. Bernardara,
\emph{Cycles, derived categories, and rationality},  in \emph{Surveys on Recent Developments in Algebraic Geometry}, Proceedings of Symposia in Pure Mathematics 95, 199-266 (2017).

\bibitem{BLMS}
A. Bayer, M. Lahoz, E. Macr\`i, P. Stellari,
\emph{Stability conditions on Kuznetsov components}, (Appendix joint also with X. Zhao), arXiv:1703.10839.

\bibitem{beauville-donagi}
A.\ Beauville, and R.\ Donagi,
\emph{La vari\'et\'e des droites d'une hypersurface cubique de dimension 4}, C.R. Acad. Sc. Paris,
S\'erie I, {\bf 301} (1985), 703--706.  

\bibitem{beauville-coble}
A. Beauville, \emph{The Coble hypersurfaces}, 
C. R. Math. Acad. Sci. Paris 337 (2003), no. 3, 189--194. 


\bibitem{vlad-bi}
V. Benedetti, 
Bisymplectic Grassmannians of planes, arXiv:1809.10902.

\bibitem{bondal-orlov-archive} 
A. Bondal, D. Orlov, \emph{Semiorthogonal decomposition for algebraic varieties},
alg-geom/9506012
 
\bibitem{dpfmr} P.~De Poi, D.~Faenzi, E.~ Mezzetti, K.~ Ranestad, \emph{Fano congruences 
of index 3 and alternating 3-forms}, Ann. Inst. Fourier 67 (2017), no. 5, 2099–2165.

\bibitem{debarre-voisin}
O. Debarre, C. Voisin, \emph{Hyper-K\"ahler fourfolds and Grassmann geometry}, J. reine angew. Math. 649 (2010), 63-87.

\bibitem{eg2}
 E.~Fatighenti, and G.~Mongardi.
\newblock A note on a Griffiths-type ring for complete intersections in Grassmannians.
\newblock {\em arXiv:1801.09586} 

\bibitem{fm} E.~Fatighenti, G.~Mongardi, \emph{Fano varieties of K3 type and IHS manifolds}, arXiv:1904.05679.

\bibitem{fik-griffiths}
D. Favero, A. Iliev, and L. Katzarkov,
\emph{On the Griffiths groups of Fano manifolds of Calabi-Yau Hodge type},
Pure Appl. Math. Q. 10 (2014), no. 1, 1–55. 

\bibitem{gsw}
L. Gruson, S. Sam, J. Weyman, 
\emph{Moduli of abelian varieties, Vinberg θ-groups, and free resolutions}, in 
Commutative algebra, 419--469, Springer 2013. 

\bibitem{hassett:special-cubics}
B.\ Hassett,
\emph{Special cubic fourfolds},
Compos. Math. {\bf 120} (2000), no.\ 1, 1--23.

\bibitem{Huybook} D.~Huybrechts, {\em Fourier-Mukai transforms in algebraic geometry}, Oxford Mathematical Monographs.
The Clarendon Press, Oxford University Press, Oxford, 2006.

\bibitem{huybrechts-rennemo}
D. Huybrechts and J. Rennemo,  \emph{Hochschild cohomology versus the Jacobian ring, and the Torelli theorem forcubic fourfolds}, Algebr. Geom., 6(1): 76–99, 2019.

\bibitem{imcy}
A.~Iliev and L.~Manivel.
\newblock  {\em Fano manifolds of Calabi - Yau Hodge type},
\newblock  Journal of Pure and Applied Algebra, 219(6), 2225-2244.

\bibitem{Jiang-Leung}
Q. Jiang, N. C. Leung,
\emph{Derived category of projectivization and flops},
arXiv:1811.12525.


\bibitem{kuz:4fold}
A.~Kuznetsov, \emph{Derived categories of cubic fourfolds}, Cohomological and
  geometric approaches to rationality problems, Progr. Math., vol. 282,
  Birkh\"{a}user Boston, 2010, p.~219--243.

\bibitem{kuz-lefschetz} 
A. Kuznetsov,
{\em Lefschetz decompositions and categorical resolutions of singularities},
Selecta Math. (N.S.) 13 (2008), no. 4, 661–696.

\bibitem{kuz} 
A. Kuznetsov,  {\em Exceptional collections for Grassmannians of isotropic lines}, 
Proc. Lond. Math. Soc. (3) 97 (2008), no. 1, 155-182.

\bibitem{kuz-hochsch}
A. Kuznetsov, \emph{Hochschild homology and semiorthogonal decompositions},	preprint math.AG/0904.4330 

\bibitem{kuz2}
A.~Kuznetsov,
\emph{On K\"uchle varieties with Picard number greater than 1},
Izvestiya: Mathematics 79.4 (2015): 698.

\bibitem{kuz-fractional}
A. Kuznetsov, \emph{Calabi-Yau and fractional Calabi-Yau categories},
J. Reine und Angewandte Mathematik, Published Online: 2017-03-02, DOI: https://doi.org/10.1515/crelle-2017-0004. 

\bibitem{kuz:ICM2014}
A.~Kuznetsov, \emph{Semiorthogonal decompositions in algebraic
geometry}, preprint arXiv:1404.3143; Proceedings of ICM-2014.

\bibitem{leuschke}
G. Leuschke, \emph{
Non-commutative crepant resolutions: scenes from categorical geometry}, 
in Progress in commutative algebra 1, 293--361, de Gruyter 2012. 

\bibitem{Li-Per-Xiao}
C. Li, L. Pertusi, X. Zhao, \emph{Twisted cubics on cubic fourfolds and stability conditions}, arXiv:1802.01134.

\bibitem{lunts}
V. Lunts, \emph{
Categorical resolution of singularities},
J. Algebra 323 (2010), 2977--3003. 

\bibitem{schubert}
L. Manivel, {\em Symmetric functions, Schubert polynomials and degeneracy loci}, SMF/AMS Texts and Monographs 6, 2001. 

\bibitem{mukai}
S.~Mukai, {\em Fano 3-folds}, in Complex projective geometry (Trieste, 1989/Bergen, 1989), 255--263,
London Math. Soc. Lecture Note Ser., 179, Cambridge 1992. 

\bibitem{orlovprojbund}
D.O. Orlov, \emph{Projective bundles, monoidal transformations and derived categories of coherent sheaves}, Russian
Math. Izv. \textbf{41} (1993), 133--141.

\bibitem{rennemo-segal}
J.V. Rennemo, E. Segal, \emph{Hori-Mological Projective Duality}, To appear in Duke Math. J.

\bibitem{snow}
Snow D., {\em 
Cohomology of twisted holomorphic forms on Grassmann manifolds and quadric hypersurfaces},
Math. Ann. {\bf 276} (1986), 159--176. 

\bibitem{sommese}
Sommese A.J., {\em Submanifolds of Abelian varieties},
Math. Ann. 233 (1978), no. 3, 229--256. 

\bibitem{voisin-book}
C. Voisin, {\em Th\'eorie de Hodge et g\'eom\'etrie alg\'ebrique complexe}. Cours Sp\'ecialis\'es {\bf 10}, Soci\'et\'e Math\'ematique de France, Paris, 2002.

\end{thebibliography}
\end{document}